\documentclass{article}
\usepackage[lang = british]{ems-jems}

\newcommand{\bA}{{\mathfrak A}}
\newcommand{\bB}{{\mathfrak B}}
\newcommand{\bC}{{\mathfrak C}}
\newcommand{\bD}{{\mathfrak D}}
\newcommand{\struct}[1]{\mathfrak{#1}}
\DeclareMathOperator{\plus}{plus}
\newcommand{\bBB}{{\mathfrak B_\infty \pi_\infty}}
\newcommand{\bBP}{\mathbf B_\infty \pi_\infty}
\newcommand{\bBBn}{{\mathfrak B_n \pi_\infty}}

\newcommand{\bBBBn}{{\mathbf B_n \pi_\infty}}
\newcommand{\bM}{{\mathfrak M}}
\newcommand{\bN}{{\mathfrak N}}

\newcommand{\bP}{{\mathfrak P}}
\newcommand{\bQ}{{\mathfrak Q}}

\newcommand{\bT}{{\mathfrak T}}

\DeclareMathOperator{\Pol}{Pol}

\DeclareMathOperator{\Csp}{CSP}
\DeclareMathOperator{\Two}{Two}

\DeclareMathOperator{\QNU}{QNU}
\DeclareMathOperator{\WNU}{WNU}
\DeclareMathOperator{\NAE}{NAE}
\DeclareMathOperator{\QHM}{QHM}
\DeclareMathOperator{\QJ}{QJ}
\DeclareMathOperator{\pr}{pr}

\newtheorem{theorem}{Theorem}
\numberwithin{theorem}{section}
\newtheorem{lemma}[theorem]{Lemma}
\newtheorem{proposition}[theorem]{Proposition}
\newtheorem{corollary}[theorem]{Corollary}
\newtheorem{conjecture}{Conjecture}
\newtheorem{definition}{Definition}
\numberwithin{definition}{section}

\begin{document}

\title{The lattice of clones of self-dual operations collapsed}
\titlemark{The lattice of clones of self-dual operations collapsed}

\emsauthor{1}{Manuel Bodirsky}{M.~Bodirsky}

\emsauthor{2}{Albert Vucaj}{A.~Vucaj}

\emsauthor{3}{Dmitriy Zhuk}{D.~Zhuk}

\emsaffil{1}{Institute of Algebra, Technische Universit\"at Dresden, Dresden, Germany \email{manuel.bodirsky@tu-dresden.de}}

\emsaffil{2}{Institute of Algebra, Technische Universit\"at Dresden, Dresden, Germany \email{albert.vucaj@tu-dresden.de}}

\emsaffil{3}{Department of Mechanics and Mathematics, Lomonosov Moscow State University, Moscow, Russia \email{Zhuk@intsys.msu.ru}}



\begin{abstract}
We prove that there are continuum many clones on a three-element set even if they are considered up to \emph{homomorphic equivalence}. The clones we use to prove this fact are clones consisting of  \emph{self-dual operations}, i.e., operations that preserve the relation $\{(0,1),(1,2),(2,0)\}$.
However, there are only countably many such  clones when considered up to equivalence with respect to \emph{minor-preserving maps} instead of clone homomorphisms.  
We give a full description of the set of clones of self-dual operations, ordered by the existence of minor-preserving maps. Our result can also be phrased as a statement about structures on a three-element set: we give a full description of the structures containing the relation $\{(0,1),(1,2),(2,0)\}$, ordered by primitive positive constructability, because  there is a minor-preserving map from the polymorphism clone of a finite structure $\bA$ to the polymorphism clone of a finite structure $\bB$ if and only if there is a primitive positive construction of $\bB$ in $\bA$.
\end{abstract}

\maketitle

\begin{abstract}
We prove that there are continuum many clones on a three-element set even if they are considered up to \emph{homomorphic equivalence}. The clones we use to prove this fact are clones consisting of  \emph{self-dual operations}, i.e., operations that preserve the relation $\{(0,1),(1,2),(2,0)\}$.
However, there are only countably many such  clones when considered up to equivalence with respect to \emph{minor-preserving maps} instead of clone homomorphisms.  
We give a full description of the set of clones of self-dual operations, ordered by the existence of minor-preserving maps. Our result can also be phrased as a statement about structures on a three-element set: we give a full description of the structures containing the relation $\{(0,1),(1,2),(2,0)\}$, ordered by primitive positive constructability, because  there is a minor-preserving map from the polymorphism clone of a finite structure $\bA$ to the polymorphism clone of a finite structure $\bB$ if and only if there is a primitive positive construction of $\bB$ in $\bA$.
\end{abstract}


\section{Introduction}
Post~\cite{Post} classified all clones of operations over the Boolean domain $\{0,1\}$; there are only countably many. 
By the famous result of Geiger~\cite{Geiger} and Bodnar\v{c}uk, Kalu\v{z}nin, Kotov and Romov~\cite{BoKaKoRo}, Post's result might also be viewed as a classification of all structures with domain $\{0,1\}$ up to \emph{primitive positive interdefinability}. Yanov and Muchnik~\cite{YanovMuchnik} showed that already over the three-element set $\{0,1,2\}$ there are $2^\omega$ many operation clones. Subsequent research in universal algebra therefore focussed on understanding particular aspects of clone lattices on finite domains, for example on the description of maximal clones~\cite{RosenbergMaximal} or minimal clones~\cite{MinClones,Csakany84,Rosenberg}. 

One might still hope to classify all operation clones on finite domains up to some equivalence relation so that equivalent clones share many of the properties that are of interest in universal algebra. Perhaps the most important equivalence relation on clones is \emph{homomorphic equivalence}: two clones 
$\mathbf C$ and $\mathbf D$ are called \emph{homomorphically equivalent} if there exists a clone homomorphism from $\mathbf C$ to $\mathbf D$ and vice versa. An attractive feature of homomorphic equivalence is that it also relates clones on different domains.  
The homomorphism order on clones has been studied intensively by Garcia and Taylor~\cite{GarciaTaylor}; it is closely related to the study of Mal'cev conditions and therefore at the heart of universal algebra. 

Every operation clone on a finite domain $B$ is the \emph{polymorphism clone} of some relational structure 
$\mathfrak B$ with domain $B$, denoted $\Pol({\mathfrak B})$. 
If ${\mathfrak A}$ and ${\mathfrak B}$ are finite structures, then there is a clone homomorphism from $\Pol({\mathfrak B})$ to $\Pol({\mathfrak A})$ if and only if ${\mathfrak A}$ has a \emph{primitive positive interpretation} in ${\mathfrak B}$. 
Note that homomorphic equivalence of clones is strictly coarser than \emph{clone isomorphism}: 
for example the two clones 
\begin{align*}
    &\Pol\Big(\{0,1\};\{(0,0,1),(0,1,0),(1,0,0)\}\Big), \mbox{ and } 
    \\&\Pol\Big(\{0,1\};\{0,1\}^3 \setminus \{(0,0,0),(1,1,1)\}\Big)
\end{align*} 
are homomorphically equivalent, but not isomorphic (the first clone has only one unary operation, while the latter has two). 

We will prove that already over a three-element set, there are $2^\omega$ many clones up to homomorphic equivalence (Corollary~\ref{cor:cont}); we are not aware of a reference in the literature for this fundamental fact. 
In our proof we use results about clones of  
\emph{self-dual operations} (see Section~\ref{sec:duals}), i.e., operations that preserve the relation 
\[C_3 \coloneqq \{(0,1),(1,2),(2,0)\}.\] Marchenkov~\cite{Marchenkov83} proved that over $\{0,1,2\}$ there are uncountably many clones of self-dual operations, and Zhuk~\cite{Zhuk15} presented a complete description of their lattice. 
This lattice has a remarkably rich structure (see Figure~\ref{fig:Dima}) and is large in the sense that ${\mathbf C}_3 \coloneqq 
\Pol(\{0,1,2\};C_3)$ is one of the 18 \emph{maximal} clones on three elements~\cite{Jablonskij}. The cardinality of the lattice of clones contained in a fixed 
maximal clone on three elements is uncountable with the exception of one maximal clone, i.e., the clone of all linear functions
(\cite{DemetrovicsHannak,Marchenkov83}; see 
Theorem 8.2.2  in~\cite{Lau}).
To the best of our knowledge ${\mathbf C}_3$ is the only maximal clone on three elements so that the lattice of its subclones is uncountable and completely described.

Motivated by research on the complexity of  constraint satisfaction problems, a coarser equivalence relation on the class of clones on a finite set has been introduced recently~\cite{wonderland}. 
We call two clones $\mathbf C$ and ${\mathbf D}$ \emph{minor equivalent} if 
there exists a minor-preserving map\footnote{Minor-preserving maps are also called \emph{minion homomorphisms}; since we do not define minions here, we stick with minor-preserving maps.} 
from ${\mathbf C}$ to ${\mathbf D}$ and vice versa. Informally, a minor-preserving map from 
$\mathbf C$ to ${\mathbf D}$ is a function that preserves arities and composition with projections;
unlike clone homomorphisms, it need not preserve the projections or general composition; a formal definition can be found in Section~\ref{sect:minor-pres}. If $\mathbf D$ is a clone over a finite domain, then it is easy to see that there is a minor-preserving map from $\mathbf C$ to $\mathbf D$ if and only if every finite set of height-one identities (also called \emph{minor condition}; see Section~\ref{sect:minor-cond}) that holds in ${\mathbf C}$ also holds in $\mathbf D$.
We will see that in the context 
of classifying clones on finite sets up to minor equivalence we may even focus on \emph{idempotent} minor conditions (also known as \emph{idempotent strong linear Mal'cev conditions}).
 
If ${\mathfrak A}$ and ${\mathfrak B}$ are finite structures and $\Pol({\mathfrak B})$ has a minor-preserving map to $\Pol({\mathfrak A})$, then 
there exists a so-called \emph{primitive positive construction} of ${\mathfrak A}$ in ${\mathfrak B}$~\cite{wonderland}.
The application in constraint satisfaction is that
if ${\mathfrak A}$ and ${\mathfrak B}$ additionally have finite relational signatures then there is a (linear-time, logarithmic space) reduction from $\Csp({\mathfrak A})$ to $\Csp({\mathfrak B})$. Indeed, the recently proven complexity dichotomy for finite-domain constraint satisfaction problems can be expressed using primitive positive constructions as follows: 
$\Csp({\mathfrak B})$ is NP-complete if the structure $(\{0,1\};\{(1,0,0),(0,1,0),(0,0,1)\})$ has a primitive positive construction in ${\mathfrak B}$, and is in P otherwise~\cite{BulatovFVConjecture,ZhukFVConjecture}. Note that primitive positive constructions can be used to study unresolved open questions in finite-domain constraint satisfaction, such as the question which finite-domain CSPs are in the complexity classes L, NL, or NC~\cite{LaroseZadori}. 

Our main result is a complete description of the clones of self-dual operations on three elements up to minor equivalence (Theorem~\ref{thm:main}). There are only countably many equivalence classes (Corollary~\ref{cor:collapse}). 
Our proof is based on the description of these clones by Zhuk~\cite{Zhuk15}. For any two clones of self-dual operations that are not minor equivalent we specify an  idempotent strong linear Mal'cev condition that holds in one clone but not the other. Interestingly, all but one of the conditions that we use are well known and play an important role in universal algebra, such as the existence of \emph{minority operations}, \emph{majority operations}, \emph{Mal'cev operations}, \emph{near unanimity operations of arity $k$}, 
\emph{J\'onsson chains}, \emph{Hagemann-Mitschke chains}, \emph{binary symmetric operations}, and
ternary \emph{weak near unanimity operations}.
The only new minor condition that we need to separate  clones of self-dual operations is
\begin{align}
r(x,x,x,y) & \approx r(x,x,x,x)  \quad \text{ and } \label{eq:guarded3cyclic} \\
r(x_1,x_2,x_3,y) & \approx r(x_2,x_3,x_1,y) \nonumber
\end{align}
which we call the \emph{guarded 3-cyclic condition}. See Figure~\ref{fig:pict}. 

Bulatov~\cite{BulatovFVConjecture} in his proof of the Feder-Vardi CSP dichotomy conjecture used graphs whose edges have three different colours; in the simpler case where every 2-element subset of the domain is primitive positive definable~\cite{Bulatov-Conservative-Revisited}, 
\emph{red edges} stand for 2-element substructures with a binary symmetric polymorphism, \emph{yellow edges} stand for 2-element substructures with a majority polymorphism but no binary symmetric polymorphism, and 
\emph{blue edges} stand for 2-element substructures with a Mal'cev polymorphism, 
but no binary symmetric and no majority polymorphism. We are inspired by this colour convention when assigning colours to the elements of the lattice drawn in Figure~\ref{fig:pict}. 

\begin{figure}
\centering
\includegraphics[scale=.5]{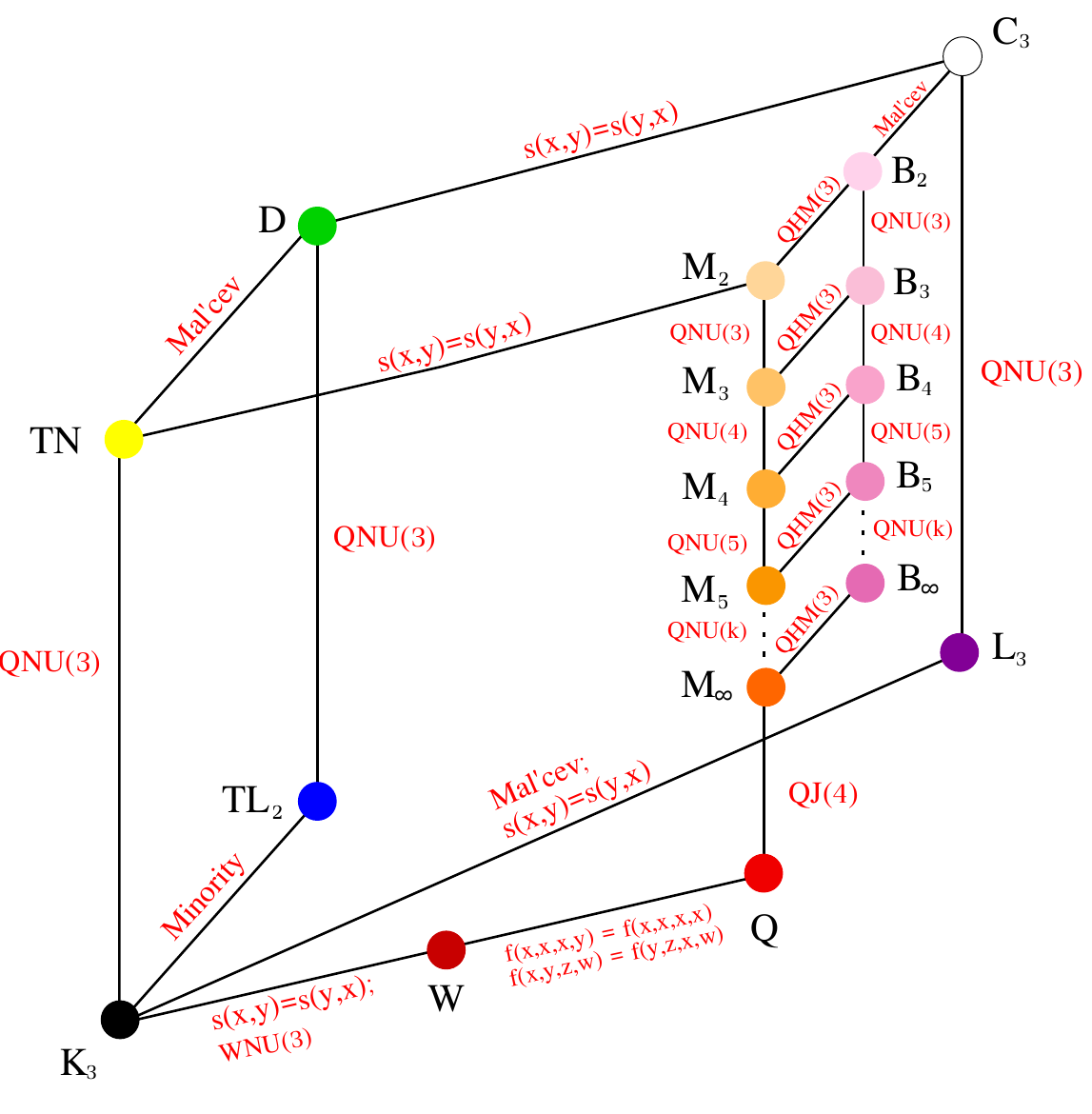}
\caption{The lattice of clones of self-dual operations up to minor equivalence.} 
\label{fig:pict}
\end{figure}

\section{Preliminaries}
\label{sect:prelims}
Let $A$ be a set. For $n \in {\mathbb N}$, we write ${\mathbf O}^{(n)}_A$ for the set
of all operations on $A$ of arity $n$ and ${\mathbf O}_A \coloneqq \bigcup_{n \in {\mathbb N}} {\mathbf O}^{(n)}_A$ for the set of all operations on $A$.
An \emph{operation clone on $A$} is a subset $\mathbf C$ of ${\mathbf O}_A$ which is closed under composition of operations and which contains all projections. If $S \subseteq {\mathbf O}_A$, then $[S]$ denotes the clone generated by $S$, i.e., the smallest clone that contains $S$. 

\subsection{Polymorphism Clones}
It is well-known and easy to see that 
every clone on a finite domain $A$ is of the form $\Pol({\mathfrak A})$ for some relational structure ${\mathfrak A}$ with domain $A$. 
Let ${\mathfrak A}$ be a structure with domain $A$ and relational signature $\tau$. 
If $A$ is finite, a well-known result~\cite{Geiger,BoKaKoRo} states that a relation $R \subseteq A^k$ is preserved by all operations of $\Pol({\mathfrak A})$ if and only if 
$R$ is \emph{primitive positive (pp-) definable} in ${\mathfrak A}$, i.e., if there exists a $\tau$-formula $\phi(y_1,\dots,y_k)$ of the form \[\exists x_1,\dots,x_n (\psi_1 \wedge \cdots \wedge \psi_m)\]
where $\psi_1,\dots,\psi_m$ are \emph{atomic}, i.e., of the form $z=z'$ for some $z,z' \in \{x_1,\dots,x_n,y_1,\dots,y_k\}$, or of the form $S(\bar z)$ for some variables $\bar z$ from  $\{x_1,\dots,x_n,y_1,\dots,y_k\}$ and $S \in \tau$, 
which \emph{defines} $R$, i.e.,
\[R = \{(a_1,\dots,a_k) \mid \bA \models \phi(a_1,\dots,a_k) \};\]
in this case $\phi$ is called a \emph{pp-definition} of $R$ in $\bA$.

\subsection{Clone Homomorphisms} 
Let ${\mathbf C}$ and ${\mathbf D}$ be operation clones. A function from ${\mathbf C}$ to ${\mathbf D}$ is called a \emph{clone homomorphism} if it preserves the arity of the operations, composition,  and maps the $i$-th projection of arity $n$ in ${\mathbf C}$ to the $i$-th projection of arity $n$ in ${\mathbf D}$. The existence of homomorphisms between polymorphism clones can be characterised in terms of \emph{primitive positive interpretations} of structures. 

\begin{definition}
Let $\bA$ and $\bB$ be structures. A \emph{primitive positive (pp) interpretation of $\bA$ in $\bB$ of dimension $d \in {\mathbb N}$} is a partial surjective map $\Gamma$ from $B^d$ to $A$ such that for every relation $R$ defined by an atomic formula over $\bA$ the relation  
$\Gamma^{-1}(R)$ is pp-definable in $\bB$ (a $k$-ary relation on $B^n$ is regarded as a $kn$-ary relation on $B$). 
\end{definition}

The following theorem is based on Birkhoff's theorem. 

\begin{theorem}[see, e.g., Corollary 6.5.16 in~\cite{Book}]\label{thm:pp-interpret}
Let ${\mathfrak A}$ and ${\mathfrak B}$ be relational structures with finite domains. Then the following are equivalent. 
\begin{itemize}
\item there is a clone homomorphism from $\Pol({\mathfrak B})$ to $\Pol({\mathfrak A})$. 
\item ${\mathfrak A}$ has a primitive positive interpretation in ${\mathfrak B}$. 
\end{itemize}
\end{theorem}

\subsection{Minor-Preserving Maps}
\label{sect:minor-pres}
If $f \colon A^n \to A$ and $\pi \colon \{1,\dots,n\} \to \{1,\dots,r\}$, then $f_\pi \colon A^r \to A$ is given by
\[ f_\pi(a_1,\dots,a_r) \coloneqq f(a_{\pi(1)},\dots,a_{\pi(n)}).\]
Let $\mathbf C$ and $\mathbf D$ be clones and let $\xi \colon \mathbf C \rightarrow \mathbf D$ be a function that preserves arities. We say that $\xi$ is \emph{minor-preserving} if,
for every $n$-ary operation $f \in \mathbf C$ and every $\pi\colon \{1,\dots,n\} \rightarrow \{1,\dots,r\}$, it holds
\begin{equation*}
\xi(f_\pi) = \xi(f)_\pi. 
\end{equation*}
The existence of minor-preserving maps between polymorphism clones can be characterised in terms of \emph{primitive positive (pp) constructions}. 
Let $\tau$ be a relational signature. 
Two relational $\tau$-structures $\bA$ and $\mathfrak B$ are \emph{homomorphically equivalent} if there exists a homomorphism from $\bA$ to $\mathfrak B$ and vice-versa. We say that $\bA$ is a \emph{pp-power (of dimension $d \in {\mathbb N}$)} of $\bB$ if it is isomorphic to a structure with domain $B^d$ whose relations are pp-definable from $\bB$.
We say that $\bA$ has a \emph{pp-construction} in $\bB$ if $\bA$ is homomorphically equivalent to a pp-power of $\bB$.
The connection between pp-constructability and
minor-preserving maps is given by the following theorem. 

\begin{theorem}[\cite{wonderland}]\label{thm:wonderland}
Let $\bA$ and $\mathfrak B$ be finite relational structures. Then the following are equivalent:
\begin{enumerate}
	\item $\bA$ has a pp-construction in $\bB$.
	\item there exists a minor-preserving map from $\Pol(\bB)$ to $\Pol(\bA)$.
\end{enumerate}
\end{theorem}

If ${\mathbf C},{\mathbf D}$ are clones, we write 
${\mathbf C} \leq_{\operatorname{minor}} {\mathbf D}$ if there exists a minor-preserving map from ${\mathbf C}$ to ${\mathbf D}$. 
Clearly, the relation $\leq_{\operatorname{minor}}$ is transitive; it follows from Theorem~\ref{thm:wonderland} that pp-constructability is transitive as well. 

\subsection{Minor conditions}
\label{sect:minor-cond}
A \emph{minor condition} is a finite set of expressions of the form $s \approx t$ where
$s$ and $t$ are terms with exactly one function symbol (sometimes, such expressions are also called \emph{height-one identities}). For example, \[\{f(x,y) \approx f(y,x)\}\] 
(sometimes just written $f(x,y) \approx f(y,x)$ without brackets) is a minor condition while $\{f(x,f(y,x)) \approx f(x,y)\}$ is not since the term on the left contains two function symbols. 
More examples of concrete minor conditions can be found in Section~\ref{sect:mnor}. 
A clone $\mathbf C$ \emph{satisfies} a minor condition if every function symbol can be instantiated by a function from $\mathbf C$ such that the resulting 
identities hold for every assignment of values from the domain of ${\mathbf C}$ to the variables in the terms. 
Clearly, if $\mathbf C$ satisfies a minor condition,
and $\xi \colon {\mathbf C} \to \mathbf D$ is a minor-preserving map, then $\mathbf D$ satisfies the same minor condition. Conversely, a simple compactness argument shows that if every minor condition that holds in $\mathbf C$ also holds in $\mathbf D$ and $\mathbf D$ is a clone on a finite domain then there exists a minor-preserving map from $\mathbf C$ to $\mathbf D$ (see, e.g.,~\cite{Book}).

\subsection{Notation}
Let $A$ be a set and $t  = (t_1,\dots,t_k) \in A^k$. 
For $I = \{i_1,\dots,i_n\} \subseteq \{1,\dots,k\}$
with $i_1<\dots<i_k$ we write $\pr_I(t)$ for the tuple $(t_{i_1},\dots,t_{i_n})$. For $R \subseteq A^k$ we write 
$\pr_I(R) \coloneqq \{\pr_I(t) \mid t \in R\}$ for the projection of $R$ to the indices from $I$. If $I = \{i\}$ for $i \in \{1,\dots,k\}$ 
then $\pr_i(R)$ denotes $\pr_I(R)$. 
A relation $R \subseteq A^k$ is called \emph{subdirect} if 
$\pr_i(R) = A$ 
for every $i \in \{1,\dots,k\}$. 
We use the symbol $+_3$ for addition modulo $3$ and $+_2$ for addition modulo $2$. 

An important notational convention in this article is that if $f \colon A^k \to A$ is an operation on $A$ and $t_1,\dots,t_k \in A^m$, then 
$f(t_1,\dots,t_k)$ denotes the tuple in $A^m$ obtained from applying $f$ componentwise. 
We also use the convention that if $f \colon A \to B$ is a function and $S \subseteq A$, then $f(S)$ denotes the set $\{f(a) \mid a \in S\}$ and $f^{-1}(S)$ denotes the set $\{f^{-1}(a) \mid a \in S\}$. Note that the two conventions can be combined. 

\section{A continuum of clones on three elements up to homomorphic equivalence}
The following relations are defined on the set $\{0,1,2\}$
\begin{align} 
C_3 & \coloneqq \{(0,1),(1,2),(2,0)\}, \\
R^=_3 & \coloneqq \{(x,y,z) \mid x \in \{0,1\} \wedge (x=0 \Rightarrow y=z) \}\label{def:relatioR3}, \\
B_2 \; & \coloneqq \{(1,0),(0,1),(1,1)\}. 
\end{align}
Marchenkov~\cite{Marchenkov83}
proved that there are $2^\omega$ many distinct operation clones between 
\begin{align*}
& {\mathbf W} \coloneqq \Pol(\{0,1,2\};C_3,R^=_3)\\
\text{ and } \quad &
{\mathbf B}_2 \coloneqq \Pol(\{0,1,2\};C_3,B_2); 
\end{align*}
our terminology is a simplified version of the terminology used in~\cite{Zhuk15}. 
Note that 
\begin{itemize}
    \item 

The relation $C_3$ and the relation $C_3$ with permuted variables can be written as 
$y = x +_3 1$, 
and 
$y = x +_3 2$, respectively;

so we freely use the terms $x +_3 1$, and $x +_3 2$ in primitive positive definitions over structures that contain the relation $C_3$. 

\item ${\mathbf B}_2 \subseteq {\mathbf W}$, because $B_2$ is pp-definable in the structure $(\{0,1,2\};C_3,R^=_3)$ by the formula 
\[ \exists z_1,z_2 \big (R^=_3(x,y,z_1) \wedge R^=_3(y,x,z_2) \wedge C_3(z_1,z_2) \big). 
\]
\end{itemize}

\begin{figure}
    \centering
    \includegraphics[scale=0.2]{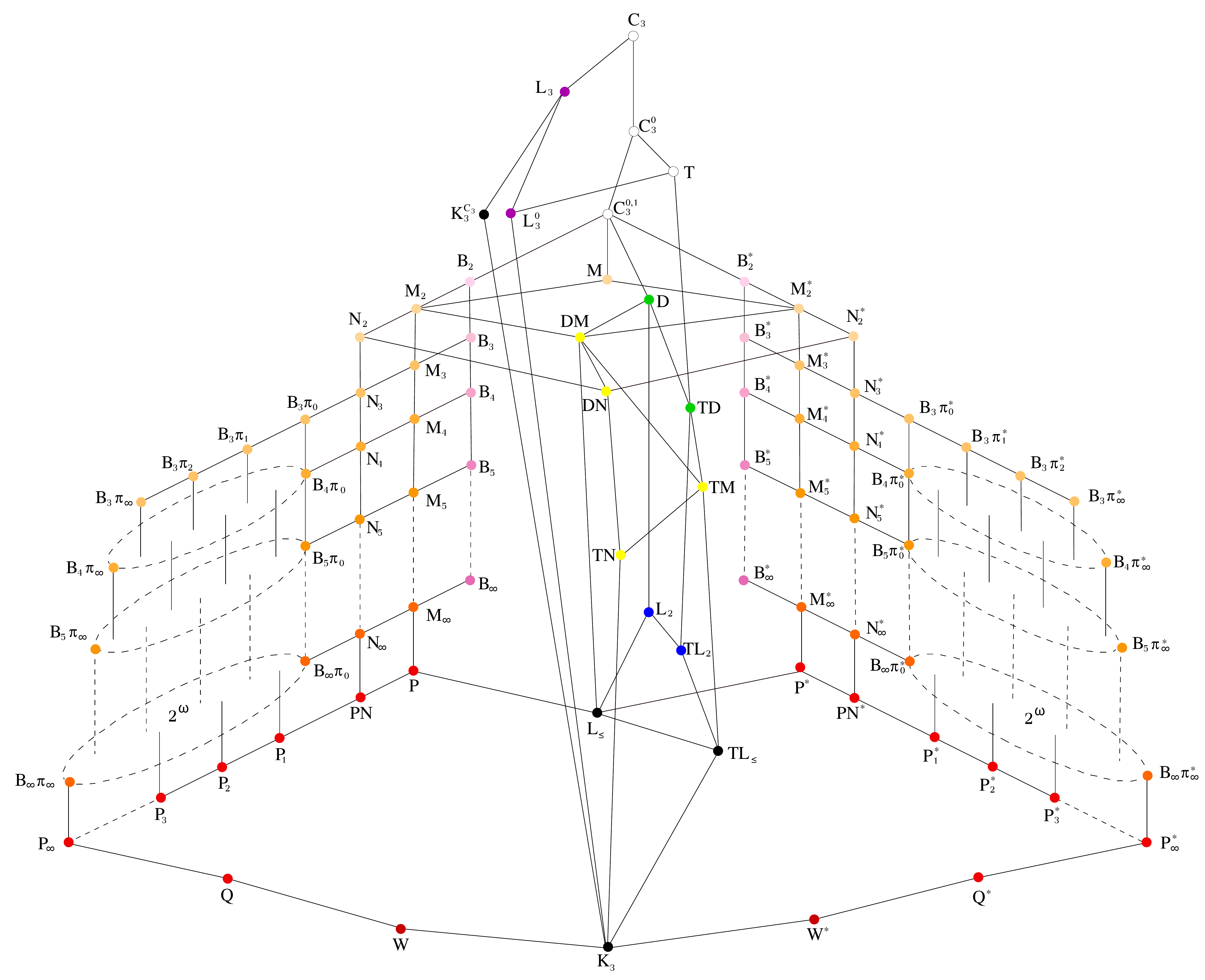}
    \caption{The lattice of clones of self-dual operations ordered by inclusion. 
    }
    \label{fig:Dima}
\end{figure}

We prove that there are $2^\omega$ many of these clones even when considered up to homomorphic equivalence (Corollary~\ref{cor:cont}). Using the facts mentioned earlier
this follows from the following theorem.

\begin{theorem}\label{thm:non-collapse}
Let $\bA$ and $\bB$ be structures such that 
$\Pol(\bA)$ and $\Pol(\bB)$ contain
$\bf W$
and are contained in 
${\bf B}_2$. 
If there is a clone homomorphism from $\Pol(\bB)$ to $\Pol(\bA)$, 
then $\Pol(\bB) \subseteq \Pol(\bA)$. 
\end{theorem}

In the proof of this theorem we need the fact that $\mathbf W$ contains the so-called \emph{paper-scissor-stone operation}, which we denote by $\vee_3$,  defined by 
\begin{align*}
\vee_3(x,y) \coloneqq \begin{cases}
x & \text{ if } x=y \\
1 & \text{ if } \{x,y\} = \{0,1\} \\
2 & \text{ if } \{x,y\} = \{1,2\} \\
0 & \text{ if } \{x,y\} = \{0,2\}. 
\end{cases}
\end{align*}
We also need the following lemma which states that if a subdirect relation has a pp-definition in 
$(\{0,1,2\};C_3,R^=_3)$, then we can also find a pp-definition without existential quantifiers and without using $R^=_3$. 

\begin{lemma}[\cite{Zhuk15}, Lemma 17] 
\label{lem:Zhuk}
Suppose that $R \subseteq \{0,1,2\}^n$ is a subdirect relation which is preserved by $\vee_3$. Then $R$ can be defined by a conjunction of atomic formulas over ${\mathfrak C}_3 \coloneqq (\{0,1,2\};C_3)$. 
\end{lemma}
\begin{proof}[Proof of Theorem~\ref{thm:non-collapse}]
By Theorem~\ref{thm:pp-interpret}, 
the structure ${\mathfrak A}$ has a pp-interpretation $\Gamma$ of dimension $d$ in ${\mathfrak B}$, for some $d \geq 1$; choose $\Gamma$ such that $d$ is smallest possible. 
Let $S \coloneqq \Gamma^{-1}(C_3) \subseteq B^{2d}$. 
Let $I \subseteq \{1,\dots,2d\}$ be the set of all $i$
such that $\pr_i(S) = \{0,1,2\}$. Note that $S' \coloneqq \pr_I(S)$ is subdirect. 
By Lemma~\ref{lem:Zhuk} the relation $S'$ can be defined by a conjunction of atomic formulas over $(\{0,1,2\};C_3)$. We distinguish two cases.

\medskip 
\textbf{Case 1.} $S' = \{0,1,2\}^{\vert I\vert}$ (Note that this includes the case that $I = \emptyset$). We will prove that this case is impossible. Choose $t = (t_1,\dots,t_{2d}) \in S$ such that $\pr_I(t) = (0,\dots,0)$. We know that $\Gamma(t) \in C_3$.
Then $\Gamma(t) = (u,u +_3 1)$ for some $u \in \{0,1,2\}$ such that $\Gamma((t_1,\dots,t_d)) = u$ and $\Gamma((t_{d+1},\dots,t_{2d})) = u+_1 1$. 

Note that 
\begin{align*}
t' & \coloneqq (t_{d+1},\dots,t_{2d},s_1,\dots,s_d) \, \in \Gamma^{-1}((u+_3 1,u+_3 2)) \subseteq \, S  \\
\text{ and }
t'' & \coloneqq (s_1,\dots,s_d,t_1,\dots,t_d)  \, \in \Gamma^{-1}((u+_3 2,u)) \subseteq \, S.
\end{align*} 
Then for every $i \in \{1,\dots,2d\}$ we have 
$\vert \{\pr_i(t),\pr_i(t'),\pr_i(t'')\}\vert  < 3$:
\begin{itemize}
\item if $i \notin I$, then this holds by the definition of $I$;
\item if $i \in I$, then 
the choice of $t$ implies that $\pr_i(t) = t_i = 0$.
Also note that $i+d \in I$, too,
since $\pr_i(S) = \pr_i(\Gamma^{-1}(\{0,1,2\})) = \pr_{i+d}(S)$, as $C_3$ is a subdirect relation on $\{0,1,2\}$. Hence, 
$\pr_i(t') = t_{i+d} = 0$. 
\end{itemize}
Let $r \coloneqq \vee_3(\vee_3(t,t'),t'')$. 
We claim that $\Gamma(r)$ is of the form $(a,a)$ for some $a \in \{0,1,2\}$, which is a contradiction since $r \in S$ and therefore $\Gamma(r) \in C_3$. 
 To see this, note that the operation $\vee_3$ is associative and commutative when restricted to sets of size two, and hence 
\[\vee_3(\vee_3((t_1,\dots,t_d),(t_{d+1},\dots,t_{2d})),s) = \vee_3(\vee_3((t_{d+1},\dots,t_{2d}),s),(t_1,\dots,t_d)).\]
Here the left hand side is the projection of $r$ to its first $d$ coordinates, while
the right hand side is the projection of $r$ to its last $d$ coordinates. Therefore,
$\Gamma(r)$ has the form $(a, a)$ for some $a \in \{0, 1, 2\}$, as claimed.

\medskip 
\textbf{Case 2.} $S' \neq \{0,1,2\}^{\vert I\vert }$. Then  there are distinct $i,j \in I$ such that $\pr_{\{i,j\}}(S)$ is the relation $C_3$ or $=$ on $\{0,1,2\}$.
If $i,j \in \{1,\dots,d\}$ or if $i,j \in \{d+1,\dots,2d\}$, then we obtain a contradiction to the assumption that the interpretation $\Gamma$ has  smallest possible dimension.
If for example $1,2 \in I$ are such that $\pr_{\{1,2\}}(S)$ equals $=$, then $\Gamma'(a_1,\dots,a_{d-1}) := \Gamma(a_1,a_1,a_2,\dots,a_{d-1})$ is a $(d-1)$-dimensional interpretation of $\bA$ in $\bB$. If $\pr_{\{1,2\}}(S)$ equals $C_3$, then 
$\Gamma'(a_1,\dots,a_{d-1}) := \Gamma(a_1,a_1 +_3 1,a_2,\dots,a_{d-1})$ is a $(d-1)$-dimensional interpretation of $\bA$ in $\bB$.

First consider the case 
that $i \in \{1,\dots,d\}$ and that $j \in \{d+1,\dots,2d\}$;
the case that $j \in \{1,\dots,d\}$ and that $i \in \{d+1,\dots,2d\}$ can be treated similarly.
We claim that for every $a \in \{0,1,2\}$ we have $\vert \pr_i(\Gamma^{-1}(a))\vert = 1$. 
To see this, let $c = (c_1,\dots,c_d),c' = (c'_1,\dots,c'_d)$ be elements of $B^d$ be such that
$\Gamma(c) = \Gamma(c') = a$.
Choose $e \in B^d$ such that $\Gamma(e) = a+_3 1$.
Since $(a,a+_3 1) \in C_3$, the tuples $(c,e) = (c_1,\dots,c_d,e_1,\dots,e_d)$ and $(c',e) = (c'_1,\dots,c'_d,e_1,\dots, e_d)$ both belong to $\Gamma^{-1}(C_3) = S$. Hence, $(c_i,e_j),(c_i',e_j) \in \pr_{\{i,j\}}(S)$, which implies
that  $c_i=c'_i$.
We write $f(a)$ for the element of
$\pr_i \big (\Gamma^{-1}(a) \big)$; note that $f$ is a permutation of $\{0,1,2\}$.

Since $\Gamma$ is an interpretation of $\bA$ in $\bB$, we know that for every relation
$R$ of $\bA$ the relation $\Gamma^{-1}(R)$ over $B$ is pp-definable in $\bB$. The same holds for every relation $R$ that is pp-definable in $\bA$: if $\phi$ is the pp-definition of $R$ in $\bA$, then we may obtain a pp-definition of 
$\Gamma^{-1}(R)$ by replacing each atomic formula in $\phi$ by its defining formula in $\bB$; the resulting formula can be rewritten into a pp-formula over the signature of $\bB$ by moving the existential quantifiers to the front. In particular, 
\begin{center}
$(*)$ for every relation $R$ that is pp-definable in $\bA$ \\
the relation $f(R) =
\pr_i(\Gamma^{-1}(R))$  is pp-definable in $\bB$.
\end{center}
The relation $B_2$ is pp-definable in $\bA$, because we assumed that $\Pol(\bA) \subseteq {\bf B}_2$. 
Therefore, $f(B_2)$ is pp-definable in $\bB$, by $(*)$. Our assumption that ${\bf W} \subseteq \Pol(\bB)$
together with $\vee_3 \in {\bf W}$ implies that $\vee_3$  preserves every
relation is pp-definable in $\bB$. 
In particular, $\vee_3$ preserves $f(B_2)$. 
Since $\vee_3$ fails to preserve $f(B_2)$ 
if the permutation $f$ is a transposition, we conclude that $f$ is either a cyclic permutation or the identity permutation. In both cases, the graph of $f$ is pp-definable in $\bB$. 
It follows that 
for every relation $R$ of $\bA$, the relation $R = f^{-1}(f(R))$ is pp-definable in $\bB$, because $f(R)$ is pp-definable in $\bB$, by $(*)$. This proves that $\bA$ is pp-definable in $\bB$, and hence 
$\Pol(\bB) \subseteq \Pol(\bA)$.
\end{proof}

We have already mentioned that there are $2^\omega$ many clones of self-dual operations~\cite{Marchenkov83}. It is even known that the interval $[{\bf W},{\bf B}_2]$ has cardinality $2^\omega$ (\cite{Marchenkov83}; also see \cite{Zhuk15} and Figure~\ref{fig:Dima}).

\begin{corollary}\label{cor:cont}
There are $2^{\omega}$ many clones of self-dual operations with respect to homomorphic equivalence. 
\end{corollary}

\section{The Collapse}
\label{sect:collapse} 
We now describe clones  of self-dual operations with respect to $\leq_{\operatorname{minor}}$. The lattice of such clones with respect to inclusion is drawn in Figure~\ref{fig:Dima}; the clones that appear in this picture will be explained in this section. 
If two clones in the picture have the same colour then this indicates that the two clones are minor equivalent. 

\subsection{The Idempotent Reduct}
\label{sect:constants}
A finite structure $\bA$ is called a \emph{core} if every endomorphism of $\bA$ is an automorphism of $\bA$. Every finite core structure $\bA$  pp-constructs the expansion of $\bA$ by all constant unary relations~\cite{wonderland}. 
Note that every expansion of $(\{0,1,2\};C_3)$ is a core. Therefore, when working with a structure $\bA$, we will often tacitly work instead with the expansion of $\bA$ by the 1-element relations, and allow constants 0, 1, and 2 in primitive positive formulas over such formulas. 
Phrased in terms of clones, this means that every clone of the form $\mathbf C=\Pol(\bA)$, where $\bA$ is a finite core, is minor equivalent to its \emph{idempotent reduct}, i.e., to the subclone of $\mathbf C$ consisting of all operations $f$ in $\mathbf C$ that satisfy $f(x,\dots,x) \approx x$.

Let us consider the following clones:
\begin{align*}
    {\mathbf C}^0_3 &\coloneqq\Pol(\{0,1,2\};C_3,\{0\}),
    \\{\mathbf L}_3 &\coloneqq\Pol(\{0,1,2\};C_3, L_3),
    \\ {\mathbf L}^0_3 &\coloneqq\Pol(\{0,1,2\};C_3,L_3,\{0\}),
\end{align*}
where $L_3$ is the following relation
\begin{align}
   L_3 \coloneqq \{ (x_1,x_2,x_3) \mid x_1 +_3 x_2 +_3 x_3 = 0\}.\label{def:RelationL3}
\end{align}
\begin{corollary}\label{cor:idempotentReduct}
The clones ${\mathbf C}_3$ and ${\mathbf C}^0_3$ are minor equivalent; the clones ${\mathbf L}_3$ and ${\mathbf L}^0_3$ are minor equivalent.
\end{corollary}

\subsection{Duals}\label{sec:duals}

In the paper we consider dualities with respect to a cyclic permutation and to a transposition.
Formally, for a permutation $\pi$ of $\{0,1,2\}$
we define $f^{\pi}$ by
\[f^{(\pi)}(x_1,\dots,x_n) \coloneqq \pi(f(\pi^{-1}(x_1),\dots,\pi^{-1}(x_n))),\]
and say that $f^{(\pi)}$ is dual to $f$ with respect to $\pi$.
As it follows from the definition, if $\pi$ is a cyclic permutation, then 
$f = f^{\pi}$ if and only if $f\in \Pol(\{0,1,2\};C_3)$,
that is why we call such operations self-dual.

Sometimes we will need duality with respect to the transposition $\sigma \colon \{0,1,2\} \to \{0,1,2\}$ defined by $\sigma(0,1,2) = (1,0,2)$. 
In the following we write $f^*$ instead of $f^{(\sigma)}$ and 
denote ${\mathbf C}^*=\{f^*\mid f\in {\mathbf C}\}$. We call $f^*$ \emph{dual to $f$ with respect to the transposition}. For example,  
$\wedge_3$ is dual with respect to the transposition to 
the paper-scissor-stone operation $\vee_3$ whose composition table can be found below, 
and we write $(\wedge_3)^* = \vee_3$.

\begin{table}[h]
\begin{minipage}[t]{0.4\linewidth}
\centering
\begin{tabular}{l|lll}
 $\vee_3$ & \multicolumn{1}{l}{0} & \multicolumn{1}{l}{1} & \multicolumn{1}{l}{2} \\ \hline
 0   &     0    &        1     &  0\\
 1   &     1    &        1     &  2\\
 2   &     0    &        2     &  2\\
\end{tabular}
\end{minipage}
\hfill
\begin{minipage}[t]{0.5\linewidth}
\begin{tabular}{l|lll}
 $\wedge_3$ & \multicolumn{1}{l}{0} & \multicolumn{1}{l}{1} & \multicolumn{1}{l}{2} \\ \hline
 0   &     0    &        0     &  2\\
 1   &     0    &        1     &  1\\
 2   &     2    &        1     &  2\\
\end{tabular}
\end{minipage}
\end{table}

Note that $f$ preserves $R$ if and only if $f^*$ preserves 
\begin{equation}
   R^* \coloneqq \{ (\sigma(x_1),\dots,\sigma(x_n)) \mid (x_1,\dots,x_n) \in R \}.\label{eq:dual} 
\end{equation}
Also note that $g \colon {\mathbf C} \to {\mathbf C}^*$ given by $f \mapsto f^*$ is minor preserving;
this implies that ${\mathbf C}$ is minor equivalent to ${\mathbf C}^*$. 
The clones such that ${\mathbf C}^* = {\mathbf C}$ are drawn in the middle of the diagram. We refer to the middle part as the \emph{spine}; all clones that are not in the spine are either in the left wing or in the mirror symmetric right wing; the symmetry is given by taking the dual with respect to the transposition and clearly visible in Figure~\ref{fig:Dima}. 

To avoid confusion we either write ``self-dual'', meaning the duality with respect to a cyclic permutation, or ``dual with respect to the transposition''.

\subsection{Collapsing 
${\mathbf P}$ and $\mathbf Q$}
\label{sect:pq-collapse}
When discussing $\leq_{\operatorname{minor}}$, we start in the lower left and work to the top. 
Let us define the following relations: ${\mathfrak Q} \coloneqq (\{0,1,2\}; C_3,R^=_2)$
and ${\mathfrak P} \coloneqq (\{0,1,2\}; C_3,R^{\Rightarrow}_2)$, where  
\begin{align*}
R^=_2 & \coloneqq \big \{(x,y,z) \mid x \in \{0,1\} \wedge (x=0 \Rightarrow y=z \in \{0,1\} ) \big \} \\
\text{ and } R^{\Rightarrow}_2 & \coloneqq \big \{(x,y,z) \mid x,y \in \{0,1\} \wedge (x = y = 0 \Rightarrow z = 0) \big \}. 
\end{align*}
Zhuk proved that the interval between the clones 
${\mathbf Q} = \Pol({\mathfrak Q})$ and
${\mathbf P} \coloneqq \Pol({\mathfrak P})$ is a countably infinite chain of clones~\cite{Zhuk15}. Theorem~\ref{thm:collapse1} below implies that the entire chain from ${\mathbf Q}$ to ${\mathbf P}$ collapses in our poset: they are all minor equivalent (Corollary~\ref{cor:PQcollapse}). Note that for every $n \geq 2$ the relation $B_n \coloneqq \{0,1\}^n \setminus \{(0,\dots,0)\}$ has the following pp-definition in $\bP$.
\begin{align*}
B_n(x_1,\dots,x_n) \Leftrightarrow \exists u_1,\dots,u_{n-1} \, \big( & R^{\Rightarrow}_2(x_1,x_2,u_1)  \wedge C_3(u_{n-1},x_1) \\
& \wedge \bigwedge_{i \in \{2,\dots,n-1\}} R^{\Rightarrow}_2(u_{i-1},x_{i+1},u_i)
\big) 
\end{align*}
Let $\leq_2$ be the relation defined as $\{(0,0),(0,1),(1,1)\}$.
Note that $x \leq_2 y$ if and only if $R_2^{\Rightarrow}(y,y,x) \wedge R_2^{\Rightarrow}(x,x,x)$ and hence $\leq_2$ is pp-definable in $\bP$.

\begin{theorem}\label{thm:collapse1}
The structure $\bQ$ has a pp-construction in the structure $\bP$. 
\end{theorem}

\begin{proof}
We first define a fourth pp-power $\bP'$ 
of $\bP$ with domain $P' \coloneqq \{0,1,2\}^4$, and then show that there exists a homomorphism $h \colon \bP' \to \bQ$
and a homomorphism $g \colon \bQ \to \bP'$. Our intuition for defining $\bP'$ will be guided by the choice of $g$: 
\begin{align}
g(0) & \coloneqq (0,1,0,0) \nonumber \\ 
g(1) & \coloneqq (1,0,1,0) \nonumber \\
g(2) & \coloneqq (2,0,0,1). \nonumber
\end{align} 

The following relations are primitively positively definable 
over $\bP$: 
\begin{align*}
C_3^{\bP'} & \coloneqq \{(x,y) \in (P')^2 \mid C_3(x_0,y_0) \wedge x_1 = y_2 \wedge x_2 = y_3 \wedge x_3 = y_1 \},   \\
(R^=_2)^{\bP'} & \coloneqq \big \{(x,y,z) \in (P')^3 \mid 
B_3(x_1,x_2,x_3) \wedge B_3(y_1,y_2,y_3) \wedge B_3(z_1,z_2,z_3) \\
& \quad \quad \quad \quad \quad \quad \quad \quad \quad \quad \wedge  x_3 = 0 \wedge x_0 = x_2 \wedge y_3 \leq_2 x_0 \wedge z_3 \leq_2 x_0 \\
&  \quad \quad \quad \quad \quad \quad \quad \quad \quad \quad \wedge 
R^{\Rightarrow}_2(x_0,y_2,z_2) \wedge
R^{\Rightarrow}_2(x_0,z_2,y_2)\big \}. 
\end{align*}

\medskip \noindent
\textbf{Claim.} $g$ is a homomorphism from $\bQ$ to $\bP'$. 
\begin{itemize}
\item Let $(a,b) \in C_3$. Then 
$(g(a)_0,g(b)_0) = (a,b) \in C_3$. Moreover, 
$g(a)_1 = g(b)_2$, $g(a)_2 = g(b)_3$, 
and $g(a)_3 = g(b)_1$. Hence $(g(a),g(b)) \in C_3^{\bP'}$. 
\item Let $(a,b,c) \in R^=_2$. 
The definition of $g$ implies that the first three conjuncts of the definition of
$(R^=_2)^{\bP'}$ are satisfied by $g(a),g(b),g(c)$. 
Moreover, $a \in \{0,1\}$ and hence $g(a)_3 = 0$. 

Suppose that $a=0$. 
We have either $b=c=0$ or $b=c=1$ by the definition of $R^=_2$. 
Then $g(a)_0 = g(a)_2 = 0$
and $0 = g(b)_3 = g(c)_3 \leq_2 g(a)_0 = 0$. 
Moreover, the last two conjuncts
in the definition of $(R^=_2)^{\bP'}$ hold:  
if $b=c=0$ then the conclusion 
in the implication $x = y = 0 \Rightarrow z=0$ from the definition of $R^{\Rightarrow}_2$ is satisfied in each of the two conjuncts, and if $b=c=1$ then the premise in the implication $x = y = 0 \Rightarrow z=0$ is not satisfied in each of the two conjuncts;
moreover, for each conjunct the first two arguments of $R^{\Rightarrow}_2$ are from $\{0,1\}$. 

Finally, suppose that $a=1$. In this case $b$ and $c$ may take any value in $\{0,1,2\}$. Note that $g(a)_0 = g(a)_2 = 1$ and $g(b)_3,g(c)_3 \leq_2 1$.
Since $g(a)_0=1$ the last two conjuncts
in the definition of $(R^\Rightarrow_2)^{\bP'}$ hold again, because the premise in the implication of the definition of $R^=_2$ is not fulfilled and because the first argument $x_0$ of these conjuncts equals $1$. This shows that $(g(a),g(b),g(c)) \in (R^=_2)^{\bP'}$. 
\end{itemize}
Define $h \colon P' \to \{0,1,2\}$ as follows. 
\begin{align*}
h(x_0,x_1,x_2,x_3) \coloneqq \begin{cases}
0 & \text{if }  (x_1,x_2,x_3) \in \{(1,0,0),(1,0,1)\} \\
1 & \text{if }  (x_1,x_2,x_3) \in \{(0,1,0),(1,1,0)\} \\
2 & \text{if }  (x_1,x_2,x_3) \in \{(0,0,1),(0,1,1)\} \\
x_0 & \text{otherwise} 
\end{cases}
\end{align*}

\medskip  \noindent
\textbf{Claim.} $h$ is a homomorphism from $\bP'$ to $\bQ$. 
\begin{itemize}
\item Let $(a,b) \in C_3^{\bP'}$. From the definition of $C_3^{\bP'}$ it follows that for a fixed $a$
there is a unique $b$ such that $(a,b)\in C_3^{\bP'}$. It is easy to check that if $a$ is such that $(a_1,a_2,a_3)\in \NAE\coloneqq\{0,1\}^3\setminus\{(0,0,0),(1,1,1)\}$, then $b\in \NAE$ and $h$ is defined such that the tuple $(h(a),h(b))\in C_3$. Otherwise, we have $(h(a),h(b)) = (a_0,b_0)$ and the first conjunct in the definition of $C_3^{\bP'}$ implies that $(a_0,b_0)\in C_3$.
 
\item Let $(a,b,c) \in (R^=_2)^{\bP'}$.
Then $a_3 = 0$ and 
$B_3(a_1,a_2,a_3)$ implies that $a_1 = 1$ or $a_2 = 1$. Moreover, $a_0 = a_2$ implies that $a\in\{(0,1,0,0),(1,1,1,0),(1,0,1,0)\}$ and 
thus $h(a) \in \{0,1\}$. 

If $h(a) = 1$ then $R^=_2(h(a),h(b),h(c))$ holds trivially.  
If $h(a) = 0$ then $(a_0,a_1,a_2,a_3)=(0,1,0,0)$ 
by the definition of $h$ and the observations above. Hence, $b_3 = c_3 = 0$ since $b_3,c_3 \leq_2 a_0 = 0$. This implies that $h(b),h(c) \in \{0,1\}$. 
From $R^{\Rightarrow}_2(0,b_2,c_2)$ and $R^{\Rightarrow}_2(0,c_2,b_2)$ it follows
that $u\coloneqq b_2=c_2$. Now we distinguish two cases: if $u=0$, we have $B_3(b_1,0,0)$ and $B_3(c_1,0,0)$, thus $b_1=c_1=1$. Therefore, $h(b) = h(c) = 0$. If $u=1$, then it follows from the definition of $h$ that $h(b_0,b_1,1,0)=h(c_0,c_1,1,0)=1$. This concludes the proof.\qedhere
\end{itemize} 
\end{proof}

\begin{corollary}\label{cor:PQcollapse}
The clones ${\mathbf P}$ and ${\mathbf Q}$ are minor equivalent. 
\end{corollary}
\begin{proof} 
It is an immediate consequence of Theorem~\ref{thm:collapse1} and Theorem~\ref{thm:wonderland} that ${\mathbf P}  \leq_{\operatorname{minor}} {\mathbf Q}$. 
Conversely, 
${\mathbf Q}  \leq_{\operatorname{minor}} {\mathbf P}$ follows from the fact that ${\bf Q} \subseteq {\bf P}$.
\end{proof}

\subsection{Collapsing $\bBP$ and ${\mathbf M}_{\infty}$}
\label{sect:maincollapse}
Below we define the clones {${\mathbf B}_{n} \pi_\infty$}, for every $n \in \{3,\dots,\infty\}$, and ${\mathbf M}_{n}$, for every $n \in \{2,3,4,\dots,\infty\}$.  
There are $2^\omega$ many clones between $\bBP$ and ${\mathbf M}_{\infty}$. 
In this section we prove that the clones $\bBP$ and ${\mathbf M}_{\infty}$  (and therefore all the $2^\omega$ many clones between them) are minor equivalent. We obtain this result as a direct consequence of Theorem~\ref{thm:collapse2}.
The proof is similar to the proof of the minor equivalence of ${\mathbf Q}$ and ${\mathbf P}$ presented in Section~\ref{sect:pq-collapse}.

\begin{definition}
Define
\begin{align*} 
\bM_{\infty} & \coloneqq (\{0,1,2\}; C_3,\leq_2,B_2,B_3,\dots) \\
\bB_{\infty} & \coloneqq (\{0,1,2\}; C_3,B_2,B_3,\dots)
\end{align*}
For $n \in \{2,3,\dots,\infty\}$, let $\bM_n$ be the reduct of $\bM_\infty$ that contains all relations of $\bM_\infty$  of arity at most $n$. The structure $\bB_n$ is defined analogously from $\bB_\infty$. 
Define 
${\mathbf M}_n \coloneqq \Pol(\bM_n)$ and ${\mathbf B}_{n} \coloneqq \Pol(\bB_n)$.
\end{definition}

Note that this definition is compatible with the definition of ${\bf B}_2$ that was already defined earlier.
To define ${\mathbf B}_n \pi_{\infty}$ we need to introduce new relations on $\{0,1,2\}$. 
Let \[W \coloneqq \left ( \begin{matrix} 0 & 0 & 1 & 1 & 1 \\
0 & 1 & 0 & 1 & 2 \end{matrix} \right ) \subseteq \{0,1,2\}^2.\] 
Note that $W(x,y)$ holds if $x \in \{0,1\}$ and 
$x = 0$ implies $y \in \{0,1\}$. 

For $m,k \in {\mathbb N}$ and $S_1 \cup \cdots \cup S_m = \{1,\dots,k\}$, the relation  
$R_{S_1,\dots,S_m}$ consists of all tuples
$(x_1,\dots,x_m,y_1,\dots,y_k) \in \{0,1,2\}^{m+k}$
such that 
\begin{enumerate}
\item $x_1,\dots,x_m \in \{0,1\}$, 
\item for every $i \in \{1,\dots,m\}$, if $x_i = 0$ then 
$y_j \in \{0,1\}$ for every $j \in S_i$, and
\item not $x_1 = \cdots = x_m = y_1 = \cdots = y_k = 0$.
\end{enumerate}

\begin{definition}
Let $\bBB$ be the structure on $\{0,1,2\}$
with the relations $C_3$, $W$, and the relation $R_{S_1,\dots,S_m}$
for every $m,k \in {\mathbb N}$ and all $S_1,\dots,S_m \subseteq \{1,\dots,k\}$ such that $S_1 \cup \cdots \cup S_m = \{1,\dots,k\}$. 
For $n \in \{3,\dots,\infty\}$, let $\bBBn$ 
be the reduct of $\bBB$ 
that contains all relations of $\bBB$ of arity at most $n$. 
As usual, we define $\bBBBn \coloneqq \Pol(\bBBn)$.
\end{definition}

It is known that $\bBBBn  \subseteq {\mathbf M}_n$, for every $n \in \{3,4,\dots,\infty\}$; in particular ${\mathbf M}_n$ contains the generator operation of $\bBBBn$ (see \cite{Zhuk15}, Theorem~29 and Theorem~30). 

\begin{theorem}\label{thm:collapse2}
For every $n \in \{3,4,\dots,\infty\}$, the structure 
$\bBBn$ has a pp-construction in the structure $\bM_n$. \end{theorem}
\begin{proof}
As in the proof of Theorem~\ref{thm:collapse1}, we use a fourth pp-power 
$\bM'_n$ of $\bM_n$, this time with the 
signature of $\bBBn$. 
The relation $C_3^{\bM'_n}$ is defined as in the proof of Theorem~\ref{thm:collapse1}. 
 
Let $k \in {\mathbb N}$ and $S_1,\dots,S_m \subseteq \{1,\dots,k\}$ be such that $S_1 \cup \dots \cup S_m = \{1,\dots,k\}$ and $m+k \leq n$. 
Then the following relations are primitively positively definable 
over $\bM_n$. 
\begin{align*}
W^{\bM'_n} & \coloneqq \{(x,y) \in (\{0,1,2\}^4)^2 \mid B_3(x_1,x_2,x_3) \wedge B_3(y_1,y_2,y_3)   \\
& \quad \quad \wedge x_3 = 0 \wedge x_0 = x_2  \wedge y_3 \leq_2 x_0 \}   \\
R_{S_1,\dots,S_m}^{\bM'_n} & \coloneqq \big \{(x^1,\dots,x^m,y^1,\dots,y^k) 
\mid  B_{m+k}(x^1_2,\dots,x_2^m,y^1_2,\dots,y^k_2) 
\\
& \wedge  \bigwedge_{i \in \{1,\dots,m\}} 
\big (B_3(x^i_1,x^i_2,x^i_3) \wedge x_3^i = 0 \wedge x^i_0 = x^i_2   
\\ & \quad \quad \quad \quad \quad \quad 
\wedge \bigwedge_{j \in S_i}  (y^j_3 \leq_2 x_0^i) \wedge B_3(y^j_1,y^j_2,y^{j}_3) \big) \big \} 
\end{align*}

Let $g \colon \{0,1,2\} \to \{0,1,2\}^4$ and $h \colon \{0,1,2\}^4 \to \{0,1,2\}$ be defined as in the proof 
of Theorem~\ref{thm:collapse1}

\medskip \noindent
\textbf{Claim.} $g$ is a homomorphism from $\bBBn$ to $\bM_n'$. We have already verified in the proof of Theorem~\ref{thm:collapse1} that $g$ preserves $C_3$. 
\begin{itemize}
\item Let $(a,b) \in W$. By the definition of $g$ we have 
\[B_3(g(a)_1,g(a)_2,g(a)_3) \mbox{ and } B_3(g(b)_1,g(b)_2,g(b)_3).\] 
If $a \neq 0$, then $a=1$ and $g(b)_3 \leq_2 g(a)_0=g(a)_2 = 1$
since $g(b)_3 \in \{0,1\}$ by the definition of $g$. 
If $a = 0$ then $b \in \{0,1\}$. Also, $g(a)_3 = 0$
and $g(b)_3 = 0 \leq_2 g(a)_0 = g(a)_2 = 0$.
\item Let 
$(a^1,\dots,a^m,b^1,\dots,b^{k}) \in R_{S_1,\dots,S_m}$. 
Then there exists $i \in \{1,\dots,m\}$ such 
that $a^i = 1$ or $j \in \{1,\dots,k\}$ such that
$b^j = 1$. 
If $a^i = 1$ then 
$g(a^i) = (1,0,1,0)$. 
If $b^j=1$ then $g(b^j) = (1,0,1,0)$. In both cases we have \[B_{m+k}(g(a^1)_2,\dots,g(a^m)_2,g(b^1)_2,\dots,g(b^k)_2).\]
To verify the other conjuncts in the definition of $R^{\bM_n'}_{S_1,\dots,S_m}$, let $i \in \{1,\dots,m\}$. 
Clearly, $B_3(g(a^i)_1,g(a^i)_2,g(a^i)_3)$. 
Since $a^i \in \{0,1\}$ we have $g(a^i)_3 = 0$ and $g(a^i)_0 = g(a^i)_2$.  Let $j \in A_i$. Clearly, $B_3(g(b^j)_1,g(b^j)_2,g(b^j)_3)$. 
If $a^i = 1$ then $g(b^j)_3 \leq_2 g(a^i)_0 = 1$ because $g(b^j)_3 \in \{0,1\}$. 
If $a^i = 0$, then $g(a^i) = (0,1,0,0)$.
We have $g(b^j)_3 = 0 \leq_2 0 = g(a^i)_0$. It follows that 
$(g(a^1),\dots,g(a^m),g(b^1),\dots,g(b^k)) \in R^{\bM_n'}_{S_1,\dots,S_m}$.
\end{itemize}

\textbf{Claim.} $h$ is a homomorphism from $\bM_n'$ to $\bBBn$.  We have already verified in the proof of Theorem~\ref{thm:collapse1} that $h$ preserves $C_3$. 
\begin{itemize}
\item Let $(a,b) \in W^{\bM_n'}$.
Since $a_3=0$ and $B_3(a_1,a_2,a_3)$ we have that
$a_0 = a_2 \in \{0,1\}$. 
If $a_0 = 1$ then 
$h(a) = 1$, and $(h(a),h(b)) \in W$. 
If $a_0=0$, then $b_3 = 0$ since $b_3 \leq_2 a_0$. 
Then from $B_3(b_1,b_2,b_3)$ it follows that $(b_1,b_2,b_3) \in \{(1,0,0),(0,1,0),(1,1,0)\}$ and therefore $h(b) \in \{0,1\}$ and $(h(a),h(b)) \in W$. 
\item Let  $(a^1,\dots,a^m,b^1,\dots,b^k) \in R_{S_1,\dots,S_m}^{\bM_n'}$. 
We have to show that \[(h(a^1),\dots,h(a^m),h(b^1),\dots,h(b^k))\]
satisfies (1), (2), and (3) in the definition of $R_{S_1,\dots,S_m}$. 
For every $i \in \{1,\dots,m\}$ we have $B_3(a^i_1,a^i_2,a^i_3)$ and $a_3^i = 0$ and so
$(a^i_1,a^i_2,a^i_3) \in \{(1,0,0),(0,1,0),(1,1,0)\}$ and thus
$h(a^i) \in \{0,1\}$, showing (1).  

Let $j \in S_i$. 
If $h(a^i) = 0$, then $b_3^j \leq_2 a^i_0 = a^i_2 = 0$ implies
that $b^j_0 \in \{0,1\}$. Since $B_3(b^j_1,b^j_2,b^j_3)$ 
we have $h(b^j) \in \{0,1\}$, showing (2). 

To prove (3), suppose for contradiction that
\[h(a^1)=\cdots=h(a^m) = h(b^1) = \dots = h(b^k) = 0.\] 
For every $i \in \{1,\dots,m\}$ we have $a^i_3 = 0$,  
$B_3(a^i_1,a^i_2,a^i_3)$, and $a^i_0 = a^i_2$, and hence $a^i = (0,1,0,0)$. Since $b^j_3 \leq_2 a^i_0$ we obtain $b^j_3 = 0$ for all $j \in \{1,\dots,k\}$. 
Since we assumed $h(b^1) = \dots = h(b^k) = 0$, it follows that $b^j\in\{(0,0,0,0),(0,1,0,0)\}$ for every $j\in\{1,\dots,k\}$; in both cases we obtain a contradiction since $B_{m+k}(a^1_2,\dots,a^m_2,b^1_2,\dots,b^k_2)$ must hold. 
This shows (3), and thus $(a^1,\dots,a^n,b^1,\dots,b^m) \in R_{S_1,\dots,S_m}$.\qedhere
\end{itemize} 
\end{proof}

\begin{corollary}\label{cor:collapse}
The minor equivalence relation has only countably many equivalence classes of clones of self-dual operations.
\end{corollary}

\subsection{A Double Collapse}
In this section we study the minor equivalence class of ${\mathbf M}_2 = \Pol(\bM_2)$. 
Let $N$ be the relation 
\begin{align}
N \coloneqq \left ( \begin{matrix} 0 & 1 & 1 & 1 \\
0 & 0  & 1 & 2 \end{matrix} \right ) \subseteq \{0,1,2\}^2. \label{eq:N}
\end{align}
Note that $N(x,y)$ holds if $x \in \{0,1\}$ and 
$x = 0$ implies $y = 0$. 
Define
\begin{align*}
\bN_2 & \coloneqq (\{0,1,2\};C_3,N,B_2) \\
\text{ and } \quad \bM & \coloneqq (\{0,1,2\};C_3,\leq_2). 
\end{align*}
We will show that there is a pp-construction of 
$\bN_2$ in $\bM$. 
Note that the formula $N(y,x) \wedge N(x,x)$ is equivalent to $x \leq_2 y$, so 
$\bM$ also has a pp-construction in $\bN_2$ (even a pp-definition).
It will follow that ${\mathbf M} \coloneqq \Pol(\bM)$ and ${\mathbf N}_2 \coloneqq \Pol(\bN_2)$ are minor equivalent. 
Note that $\bM^* = {\bM}$, so $\bM$ is already part of the spine.

Instead of directly specifying the pp-construction of $\bM$ in $\bN_2$, 
we found it more convenient to prove this in two steps: first we show that
$\bN_2$ has a pp-construction in $\bM_2$, and then we show that $\bM_2$ has a pp-construction in $\bM$. 

\begin{lemma}
 There is a pp-construction of $\bN_2$
in $\bM_2$. 
\end{lemma}
\begin{proof}
We again use a four-dimensional pp-power $\bA$ 
of $\bM_2$ and the maps $h$ and $g$ from the proof of Theorem~\ref{thm:collapse1}. 
The structure $\bA$ has domain $A \coloneqq \{0,1,2\}^4$ and signature $\{C_3,N,B_2\}$.
The relation $C_3^{\bA}$ is defined as in the proof of  Theorem~\ref{thm:collapse1}, and 
\begin{align*}
N^{\bA} &\coloneqq \{ (x,y) \in A^2 \mid x_3=0 \wedge x_0 = x_2  \wedge B_2(x_1,x_2) \\
& \quad \quad \quad \quad \quad \, \,\,\quad \quad \quad\quad \quad  \wedge x_1\leq_2 y_1\wedge y_2 \leq_2 x_2 \wedge y_3\leq_2 x_0\} ,
\\
B^{\bA}_2 &\coloneqq \{(x,y) \in A^2 \mid B_2(x_0,y_0)\wedge x_0 = x_2 \wedge y_0 = y_2\wedge x_3 = y_3 = 0\}. 
\end{align*}

\textbf{Claim.} The map $g$ is a homomorphism from $\bN_2$ to $\bA$. \\
We have already verified in the proof of Theorem~\ref{thm:collapse1} that $g$ preserves $C_3$. 
\begin{itemize}
\item Let $(a,b) \in N$. If $a = 0$, then $b=0$, and hence $(x_0,x_1,x_2,x_3) \coloneqq g(a) = (0,1,0,0)$ and $(y_0,y_1,y_2,y_3) \coloneqq g(b) = (0,1,0,0)$
satisfies the formula from the definition of $N^{\bA}$. If $a = 1$, then putting $(x_0,x_1,x_2,x_3) \coloneqq g(a) = (1,0,1,0)$ satisfies $x_3 = 0, x_0 = x_2$, $B_2(x_1,x_2)$. Moreover, putting $(x_0,x_1,x_2,x_3) \coloneqq g(b)$ all the remaining conjuncts in the definition of $N^{\bA}$ are satisfied as well: for $x_1 \leq_2 y_1$ since
$x_1 = 0$, and for $y_2 \leq_2 x_2$ and $y_3 \leq_2 x_0$ since $y_2,y_3 \in \{0,1\}$ and $x_0=x_2=1$. 
\item Let $(a,b) \in B_2$. Then $a,b \in \{0,1\}$, and hence $(g(a),g(b))$ satisfies 
$x_0 = x_2$, $y_0=y_2$, and $x_3=y_3=0$. 
Moreover, $B_2(x_0,y_0)$ holds because otherwise $a=b=0$, contrary to the assumption that $B_2(a,b)$. 
\end{itemize} 

\textbf{Claim.} The map $h$ is a homomorphism from $\bA$ to $\bN_2$. \\
We have already verified in the proof of Theorem~\ref{thm:collapse1} that $h$ preserves $C_3$. Let $(x,y) \in N^{\bA}$.
We must have $x_3=0$ and $x_0 = x_2$
therefore $h(x) \in \{0,1\}$. Since $x_3=0$ and $B_2(x_1,x_2)$ it is impossible that $x_1=x_2=x_3$ and therefore $x \in \{(0,1,0,0),(1,0,1,0),(1,1,1,0)\}$. 
If $x \in \{(1,0,1,0),(1,1,1,0)\}$ then $h(x) = 1$ and 
$(h(x),h(y)) \in N$. If $x = (0,1,0,0)$ then $1 = x_1 \leq_2 y_1$, $y_2 \leq_2 x_2 = 0$, and $y_3 \leq_2 x_0 = 0$, and hence $h(y) = 0$. Thus, $(h(x),h(y)) = (0,0) \in N$. 
 
 Let $(x,y) \in B_2^{\bA}$. Again the conjuncts $x_3=0$ and $x_0=x_2$ imply that $h(x) \in \{0,1\}$, and similarly $h(y) \in \{0,1\}$. If $h(x) = 0$ then $x = (0,0,0,0)$. In this case the conjunct $B_2(x_0,y_0)$ implies that $y_0 \neq 0$, and therefore that $h(y) \neq 0$. So $(h(x),h(y)) \in B_2$. 
 \end{proof}

\begin{lemma}\label{lem:Mpp}
 There is a pp-construction of 
 $\bM_2$ 
 in $\bM$. 
 \end{lemma}
\begin{proof} 
Let $\bA$ be the structure with the domain $A\coloneqq \{0,1,2\}^2$ and the signature $\{C_3,\leq_2,B_2\}$  such that 
\begin{align*}
C^{\bA}_3 & \coloneqq \{ (x,y) \in A^2 \mid C_3(x_1,y_1) \wedge C_3(y_2,x_2)\}, \\
\leq_2^{\bA} & \coloneqq \{ (x,y) \in A^2 \mid x_1\leq_2 y_1 \wedge y_2\leq_2 x_2\}, \\
B^{\bA}_2 & \coloneqq \{ (x,y) \in A^2 \mid x_2\leq_2 y_1 \wedge x_1 \leq_2 x_1 \wedge y_2 \leq_2 y_2\}.
\end{align*}

Let us define the map $g \colon \{0,1,2\} \to \{0,1,2\}^2$ as follows: $g(0) \coloneqq (0,1)$, $g(1) \coloneqq (1,0)$, and $g(2) \coloneqq~(2,2)$.

\medskip 
\textbf{Claim 1:} $g$ is a homomorphism from $\bM_{2}$ to $\bA$.
Let $(a,b) \in C_3$. If $(a,b) = (0,1)$ then
$(g(a),g(b)) = ((0,1),(1,0)) \in C_3^{\bA}$. 
If $(a,b)= (1,2)$ then $(g(a),g(b)) = ((1,0),(2,2)) \in C_3^{\bA}$. If $(a,b) = (2,0)$ then $(g(a),g(b)) = ((2,2),(0,1)) \in C_3^{\bA}$. 

Now suppose that $a \leq_2 b$. If $a = b = 0$ then
\[(g(a),g(b)) = ((0,1),(0,1)) \in \; \leq_2^{\bA}\] 
and similarly for $a=b=1$. If $a = 0$ and $b=1$
then \[(g(a),g(b)) = ((0,1),(1,0)) \in \; \leq_2^{\bA}.\]

Finally, suppose that $(a,b) \in B_2$.
Then $a,b \in \{0,1\}$ and hence the entries of
$h(a)$ and $h(b)$ are from $\{0,1\}$ as well. Thus, 
the final two conjuncts in the definition of $B_2^{\bA}$ are always satisfied. 
If $a=b=1$
then $(g(a),g(b)) = ((1,0),(1,0)) \in B_2^{\bA}$
since $0 \leq_2 1$.  
If $a=1$ and $b=0$ then $(g(a),g(b)) = ((1,0),(0,1)) \in B_2^{\bA}$ since $0 \leq_2 0$. And if $a=0$ and $b=1$ then
$(g(a),g(b)) = ((0,1),(1,0)) \in B_2^{\bA}$ since $1 \leq_2 1$. 

\medskip 
Let $h \colon \{0,1,2\}^2 \to \{0,1,2\}$ be defined  by 
\begin{align*}
h(0,1) & =h(0,2)=h(2,1)=0,\\
h(0,0) & =h(1,0)=h(1,1)=1,\\
h(2,0) & =h(1,2)=h(2,2)=2.
\end{align*}
Note that $h(a,b) = a$ for every $a,b \in \{0,1,2\}$ with three exceptions: $h(2,1) = 0$, $h(0,0) = 1$, and $h(1,2) = 2$; we call $(2,1),(1,2),(0,0)$ the \emph{exceptional points} and all other points \emph{regular}. 

\textbf{Claim 2:} $h$ is a homomorphism from $\bA$ to $\bM_2$. 
Let $(a,b) \in C_3^{\bA}$. 
Then $C_3(a_1,b_1)$ and $C_3(b_2,a_2)$. 
In particular, note that $a$ equals $(2,1)$
if and only if $b$ equals $(0,0)$. In this case, $(h(a),h(b)) = (0,1) \in C_3$ and we are done. Similarly, we verify the statement if $a$ equals $(1,2)$ or $(0,0)$. 
Also note that $a$ is an exceptional point for $h$ if and only if $b$ is. If $a$ and $b$ are regular then $(h(a),h(b)) = (a_1,b_1) \in C_3$ and we are done. 

Now let $(a,b) \in \; \leq_2^{\bA}$. Then by definition 
we have $a_1 \leq_2 b_1$ and $b_2 \leq_2 a_2$. 
Note that in particular, neither $a$ nor $b$ can be
the exceptional points $(1,2)$ or $(2,1)$ for $g$. 
If $a = (0,0)$ then $b_1,b_2 \in \{0,1\}$,
hence $h(b) \in \{0,1\}$, and 
thus $h(b) \in \{0,1\}$ and $h(b) \leq_2 h(a) = 1$. 
For the regular points the verification is again immediate. 

Finally, let $(a,b) \in B_2^{\bA}$. Note that then 
$a_1,a_2,b_1,b_2 \in \{0,1\}$, hence 
$h(a),h(b) \in \{0,1\}$. Suppose that $h(b) = 0$.
By the definition of $B_2^{\bA}$, we have $a_2 \leq_2 b_1 = 0$,
so we have $h(a) = h(0,0)= 1$. This shows that $(h(a),h(b)) \in B_2$. 
\end{proof}

\subsection{The Spine}
We now discuss collapses in the spine; 
again, some of them can be proved by exhibiting pp-constructions, while in one case it was more convenient to directly exhibit a minor-preserving map (in Proposition~\ref{prop:tn-dm}). 
The following structures are at the bottom of the spine in Figure~\ref{fig:Dima}; we show that they are equivalent with respect to pp-constructability. 
\begin{align*}
    \mathfrak{TL}_\leq &\coloneqq (\{0,1,2\}; C_3,T,L_2,\leq_2) & \mathfrak{K}^{C_3}_3 &\coloneqq (\{0,1,2\};C_3, \neq)\\
        {\mathfrak L}_\leq & \coloneqq (\{0,1,2\}; C_3,L_2,\leq_2)    &  \mathfrak{K}_3 &\coloneqq (\{0,1,2\}; \neq,\{0\},\{1\},\{2\})
     \end{align*}
where 
\begin{align}
L_2  & \coloneqq \{(x_1,x_2,x_3) \in \{0,1\}^3 \mid x_1 +_2 x_2 +_2 x_3 = 0\} \label{def:L} \\
\text{ and } \quad T & \coloneqq \{(0,1),(1,0),(2,2)\}. 
\label{def:T}
\end{align}

\begin{lemma}
The structures $\mathfrak{K}_3$, $\mathfrak{K}^{C_3}_3$, $\mathfrak{TL}_{\leq}$, and $\mathfrak{L}_\leq$  pairwise pp-interpret each other (and hence in particular pp-construct each other). 
\end{lemma} 
\begin{proof}
Every finite structure has a pp-interpretation in 
$(\{0,1,2\};\neq)$ (see, e.g.,~\cite{Book}). The structure $\mathfrak{TL}_\leq$ is an expansion of $\mathfrak{L}_\leq$. Finally, observe that 
$\mathfrak{L}_\leq$ interprets primitively positively 
the structure 
\[(\{0,1\}; L_2, \leq_2,\{0\},\{1\})\] 
and that all polymorphisms of this structure are projections~\cite{Post}. Hence, $\mathfrak{L}_\leq$ pp-interprets all finite structures (see, e.g., Theorem 6.3.10 in~\cite{Book}). Similarly, note that $\mathfrak{K}^{C_3}_3$ pp-defines $(\{0,1,2\};\neq)$, which in turn pp-interprets all finite structures. Hence, so does $\mathfrak{K}^{C_3}_3$.
\end{proof}

In the lattice of clones of self-dual operations  (Figure~\ref{fig:Dima})  
the clone $\Pol({\mathfrak L}_{\leq}$) has the cover
$\Pol({\mathfrak L}_2)$ where
\[{\mathfrak L}_2 \coloneqq (\{0,1,2\};C_3,L_2)\]
and the clone $\Pol(\mathfrak{TL}_\leq)$ has the cover $\Pol({\mathfrak{TL}}_2)$ where  \[\mathfrak{TL}_2 \coloneqq (\{0,1,2\};C_3,T,L_2).\]

\begin{proposition}\label{prop:tl-l}
The structures $\mathfrak{TL}_2$ and  ${\mathfrak L}_2$
pp-construct each other. 
\end{proposition}
\begin{proof}
Since $\mathfrak{TL}_2$ is an expansion of ${\mathfrak L}_2$ 
it suffices to prove that $\mathfrak{L}_2$ pp-constructs $\mathfrak{TL}_2$. We consider the pp-power $\bA$ of ${\mathfrak L}_2$ with domain $\{0,1,2\}^2$ and the same signature as $\mathfrak{TL}_2$ whose relations are defined as follows. 
\begin{align*}
    C^{\bA}_3 & \coloneqq \{(x,y) \in A^2 \mid C_3(x_1,y_1) \wedge C_3(y_2,x_2) \} 
    \\
    T^{\bA} &\coloneqq \{(x,y) \in A^2 \mid x_1=y_2\wedge x_2=y_1 \} 
    \\
    L_2^{\bA} &\coloneqq 
   \{(x,y,z) \in A^3 \mid L_2(x_1,y_1,z_1)\wedge L_2(x_1,x_2,1) 
   \wedge L_2(y_1,y_2,1)\wedge L_2(z_1,z_2,1) \}.
\end{align*}
We prove that $\mathfrak{TL}_2$ and ${\bA}$ 
are homomorphically equivalent. Let 
$g \colon \{0,1,2\} \to \{0,1,2\}^2$
be defined by 
$g(0) \coloneqq (0,1)$, 
$g(1) \coloneqq (1,0)$, and 
$g(2) \coloneqq (2,2)$. Then $g$ is a homomorphism from $\mathfrak{TL}_2$ to ${\mathfrak A}$: the proof that $g$ preserves $C_3$ we have already seen this in the proof of Lemma~\ref{lem:Mpp}. Now suppose that 
$(x,y) \in T$. If $(x,y) = (0,1)$ then $(g(0),g(1)) = ((0,1),(1,0)) \in T^{\mathfrak{A}}$. 
For $(x,y) \in \{(1,0),(2,2)\}$ the argument is similarly straightforward. Finally, suppose that
$(x,y,z) \in L$. Then $x,y,z \in \{0,1\}^2$, and hence $g(x),g(y),g(z) \in \{(0,1),(1,0)\}$. Hence, the conjuncts
$L_2(x_1,x_2,1)$, $L_2(y_1,y_2,1)$, $L_2(z_1,z_2,1)$ in the definition of $L_2^{\bA}$ are satisfied. 
Moreover, $x +_2 y +_2 z = 0$ implies that
$g(x)_1 +_2 g(y)_1 +_2 g(z)_1 = 0$, and hence
$(g(x),g(y),g(z)) \in L_2^{\bA}$. 

Let $h \colon \{0,1,2\}^2 \to \{0,1,2\}$ be defined as 
\begin{align*}
h(x,y) \coloneqq \begin{cases}
0 & \text{if } C_3(x,y) \\
1 & \text{if } C_3(y,x) \\
2 & \text{if } x=y. 
\end{cases}
\end{align*}
We prove that $h$ is a homomorphism from ${\mathfrak A}$ to $\mathfrak{TL}_2$. 
Let $((x_1,x_2),(y_1,y_2)) \in C_3^{\bA}$.
Then $C_3(x_1,y_1)$ and $C_3(y_2,x_2)$.  
If $h(x_1,x_2) = 0$, then 
$C_3(x_1,x_2)$, and hence $C_3(y_2,y_1)$,
and therefore $h(y_1,y_2) = 1$. Hence, 
$(h(x_1,x_2),h(y_1,y_2)) \in C_3$.
The verification if $h(x_1,x_2) \in \{1,2\}$ is similarly straightforward. 
The proof that $h$ preserves $T$ is similar as well. Finally, suppose that 
\[((a_1,b_1),(a_2,b_2),(a_3,b_3)) \in L_2^{\bA}.\] 
Then $L(a_i,b_i,1)$, for $i \in \{1,2,3\}$, implies that $(a_i,b_i) \in \{(0,1),(1,0)\}$, and hence $h(a_i,b_i) \in \{0,1\}$. Note that in this case $h(a_i,b_i) = a_i$, and hence $L_2(a_1,a_2,a_3)$ implies
that $(h(a_1,b_1),h(a_2,b_2),h(a_3,b_3)) \in L$. 
\end{proof}

Define
\begin{align*}
\mathfrak{TN} & 
\coloneqq (\{0,1,2\};C_3,T,N,N^*) \\
\text{ and } \quad \mathfrak{DM} & \coloneqq (\{0,1,2\};C_3,C_2,\leq_2). 
\end{align*}
where $C_2 \coloneqq \{(0,1),(1,0)\}$
and $N^* = \{(0,0),(0,1),(1,1),(0,2)\}$ is dual with respect to transposition to the relation $N = \{(0,0),(1,0),(1,1),(1,2)\}$ defined in~(\ref{eq:N}). 
Zhuk~(\cite{Zhuk15}, Theorem~28) proved that
\begin{align*}
\mathbf{TN} \coloneqq \Pol(\mathfrak{TN}) & = [m] \\
\text{ and } \quad \mathbf{DM} \coloneqq \Pol(\mathfrak{DM}) & = [m,p]
\end{align*}
where the operations $m$ and $p$ are defined as follows: 
\begin{align}
    m(x,y,z)& \coloneqq \begin{cases}
    (x\wedge_3 y)\vee_3 (x\wedge_3 z)\vee_3 (y\wedge_3 z) & \text{ if } |\{x,y,z\}|\leq 2
    \\x & \text{ otherwise,}
    \end{cases}
    \label{def:m}
    \\p(x,y,z)& \coloneqq \begin{cases}
    x & \text{ if } |\{x,y,z\}|\leq 2
    \\ x+_3 1 & \text{ otherwise.}
    \end{cases}
    \label{def:p}
\end{align}

Rather than proving that $\mathfrak{TN}$ and
$\mathfrak{DM}$ pp-construct each other, we follow a different strategy and directly work with the respective clones and minor preserving maps. 

\begin{proposition}\label{prop:tn-dm}
The clones $\mathbf{TN}$ and $\mathbf{DM}$ are minor equivalent. 
\end{proposition}
\begin{proof}
Since $\mathbf{DM}$ contains $\mathbf{TN}$, it
suffices to find a minor-preserving map $\Xi$ from
$\mathbf{DM}$ to $\mathbf{TN}$. 
We first define a minor-preserving map $\xi$ over $\mathbf{DM}^{(3)}$, i.e., the set of all operations of arity at most three in $\mathbf{DM}$ (see \textbf{Claim~1}). Then we extend $\xi$ to a minor-preserving map $\Xi$ from $\mathbf{DM}$ to the clone of all operations on $\{0,1,2\}$, and finally 
we show that the image of $\Xi$ lies in $\mathbf{TN}$ (see \textbf{Claim~2} and \textbf{Claim~3}). 

Note that every binary operation of $\mathbf{DM}$ must be a projection: this follows by induction over the generation process from the fact that $\mathbf{DM} = [m,p]$ is generated by $m$ and $p$, since any operation obtained from the ternary operations $m$ and $p$ by identifying arguments is a projection. 
Moreover, every ternary operation of $\mathbf{DM}$ restricted to $\{0,1\}$ must be
either a projection or the ternary majority operation; again, this is easy to show by induction over the generation process. 

For $i \in \{1,2\}$, we define $\xi(\pr^2_i) \coloneqq \pr^2_i$. 
Let $f \in \mathbf{DM}^{(3)}$. Note that every operation in $\mathbf{DM}$ preserves $\{0,1\}$; if the restriction of $f$ to $\{0,1\}$ is a projection to the $i$-th argument, 
then we define $\xi(f) \coloneqq \pr^3_i$. 
Suppose now that the restriction of $f$ to $\{0,1\}$ is a majority operation. Note that for $b_1,b_2,b_3 \in \{0,1,2\}$ and $i \in \{1,2\}$ we have 
$f(b_1 +_3 i,b_2 +_3 i,b_3 +_3 i) = f(b_1,b_2,b_3) +_3 i$ since $f$ preserves $C_3$. Hence,
$f$ is fully determined by its values for $(0,1,2)$ and on $(0,2,1)$. Moreover, for every $i\in\{1,2,3\}$, let $\sigma_i\colon\{1,2,3\}\to\{1,2,3\}$ be the map that fixes $i$ and permutes the remaining two elements. For every $f\in\mathbf{DM}^{(3)}$, there exists an $i\in\{1,2,3\}$ such that $f=f_{\sigma_i}$:

\begin{itemize}
\item If $f(0,1,2) = f(0,2,1)$, then $f=f_{\sigma_1}$
and we define $\xi(f) \coloneqq m(x,y,z)$. Note that $m(x,y,z)=m_{\sigma_1}$. 
\item If $f(0,1,2) = f(0,2,1) +_3 1 = f(1,0,2)$, 
then $f=f_{\sigma_3}$ 
and we define $\xi(f) \coloneqq m(z,y,x)=m_{\sigma_2}$. 
\item If $f(0,1,2) = f(0,2,1) +_3 2 = f(2,1,0)$, then $f=f_{\sigma_2}$ 
and we define $\xi(f) \coloneqq m(y,x,z)=m_{\sigma_3}$.
\end{itemize}
For the sake of notation, let us define a map $\nu\colon\{1,2,3\}\to\{\sigma_1,\sigma_2,\sigma_3\}$ such that $\nu(1)=\sigma_1$, $\nu(2)=\sigma_3$, and $\nu(3)=\sigma_2$. Note that, if $f=f_{\sigma_i}$, then $\xi(f)=m_{\nu(i)}$.

\textbf{Claim~1:} The map $\xi$ is minor-preserving. If $f\in \mathbf{DM}^{(3)}$ is such that its restriction to $\{0,1\}$ is a projection, then it can be easily checked that $\xi(f_\pi) = \xi(f)_{\pi}$ for every $\pi \colon \{1,2,3\} \to \{1,2,3\}$. Now let us consider the case where the restriction of $f$ to $\{0,1\}$ is the majority operation. If $\pi \colon \{1,2,3\} \to \{1,2,3\}$ in not injective, then  $\pi(j)=\pi(k)=i$ for some $i,j,k\in\{1,2,3\}$; in this case, $\xi(f_{\pi}) = \pr^3_{i} = \xi(f)_\pi$. Let us finally consider the case where $\pi$ is injective. If $f\in \mathbf{DM}^{(3)}$ is such that $f=f_{\sigma_i}$ for some $i\in\{1,2,3\}$, then first applying $\pi$ we get $f_\pi = f_{\pi\circ\sigma_i}$ and then via $\xi$ we get $m_{\pi\circ\nu(i)}$. On the other side, if we first apply $\xi$ we get $m_{\nu(i)}$ which is mapped to $m_{\pi\circ\nu(i)}$ by $\pi$.

Since $\mathbf{DM}$ is defined on a set of cardinality three, the map $\xi$ naturally extends to a map $\Xi$ from $\mathbf{DM}$ to the set of all operations on 
$\{0,1,2\}$ as follows: 
for every $f \in \mathbf{DM}$ of arity $n$ and $a_1,\dots,a_n \in \{0,1,2\}$, let $f'$ be the ternary operation in $\mathbf{DM}$ defined by $f'(x_0,x_1,x_2) \coloneqq f(x_{a_1},\dots,x_{a_n})$. Then let $\Xi(f)$ 
be the $n$-ary operation on $\{0,1,2\}$ that maps
$(a_1,\dots,a_n)$ to $\xi(f')(0,1,2)$. 

The map $\Xi$ is minor-preserving by definition, so we are left with showing that $\Xi(f) \in \mathbf{TN}$ for every $f \in \mathbf{DM}$. 

\textbf{Claim~2:} The operation $\Xi(f)$ preserves $C_3$ and $T$. Observe that, since $|C_3|=|T|=3$, it is sufficient to prove the claim for ternary operations in $\mathbf{DM}$; indeed, if $f$ does not preserve $C_3$ or $T$, then there is a ternary minor of $f$ that does not preserve $C_3$ or $T$. If $g$ is a ternary operation in $\mathbf{DM}$, then $\Xi(g)= \xi(g)$ and therefore the claim holds since $\xi(g)\in\mathbf{TN}$ by the definition of $\xi$.

\textbf{Claim~3:} $\Xi(f)$ preserves $N$. It suffices to show that every 4-variable minor of $f$ preserves $N$, because $|N|=4$. 
Let $g \in \mathbf{DM}^{(4)}$ and let
$(a_1,b_1),(a_2,b_2),(a_3,b_3),(a_4,b_4) \in N$.
We may 
suppose that if $i,j \in \{1,2,3,4\}$ are distinct, then $(a_i,b_i) \neq (a_j,b_j)$. Indeed, suppose that
for $i=1$ and $j=2$ we have 
$(a_i,b_i) = (a_j,b_j)$. Then let $g'$ be the ternary minor of $g$ defined by $g'(x,y,z) \coloneqq g(x,x,y,z)$, and note that  
\begin{align*}
 \big (\Xi(g)(a_1,a_2,a_3,a_4),\Xi(g)(b_1,b_2,b_3,b_4) \big ) 
=  \big (\xi(g')(a_1,a_3,a_4),\xi(g')(b_1,b_3,b_4) \big )
\end{align*} 
and we conclude that $\Xi(g)$ preserves $N$ since 
$\xi(g') \in \mathbf{TN}$. The argument for different pairs of distinct elements $i,j \in \{1,2,3,4\}$ is similar. 
By permuting arguments of $g$, we may therefore assume without loss of generality that \[\big ((a_1,b_1),(a_2,b_2),(a_3,b_3),(a_4,b_4) \big ) = \big ((0,0),(1,0),(1,1),(1,2) \big).\]

Let $g_1,g_2 \in \mathbf{DM}$ be the ternary operations
given by 
\begin{align*}
g_1(x_0,x_1,x_2) & \coloneqq g(x_{a_1},x_{a_2},x_{a_3},x_{a_4}) = g(x_0,x_1,x_1,x_1) \\
\text{ and } 
g_2(x_0,x_1,x_2) & \coloneqq g(x_{b_1},x_{b_2},x_{b_3},x_{b_4}) = g(x_0,x_0,x_1,x_2).
\end{align*}

Since $g_1$ must be a projection to the first or second argument we obtain 
\[\Xi(g)(0,1,1,1) = \xi(g_1)(0,1,2) \in \{0,1\}.\]
To show that $(\Xi(g)(0,1,1,1),\Xi(g)(0,0,1,2)) \in N$ it therefore suffices to show that if
$\Xi(g)(0,1,1,1) = 0$, then 
$\Xi(g)(0,0,1,2) = 0$. 
Suppose that $\Xi(g)(0,1,1,1) = 0$. Since $\Xi(g)(0,1,1,1) = \xi(g_1)(0,1,2)$ and $\xi(g_1)$ is a projection, we then must have $\xi(g_1) = \pr^3_1$. 
By the definition of $\xi$, this implies that $g_1 = \pr^3_1$, hence $g(0,1,1,1)=0$. 

Note that the restriction $g_2'$ of $g_2$ to $\{0,1\}$ is a projection or a majority operation. 
Suppose for contradiction that $g_2'$ is the second projection. 
Then $g(0,0,1,0) = g'_2(0,1,0) = 1$, but $g(0,1,1,1) = 0$, a contradiction to the assumption that $g$ preserves $\leq$. 
If $g_2'$ to $\{0,1\}$ is the third projection then
$g(0,0,0,1) = g'_2(0,0,1) = 1$, which similarly leads to a contradiction. 
If $g_2'$ is the majority operation then
$g(0,0,1,1) = g'_2(0,1,1) = 1$, again leading to a contradiction. 
Hence, $g_2'$ is must be the first projection. 
It follows that $\Xi(g)(0,0,1,2) = \xi(g_2)(0,1,2) = \pr^3_1(0,1,2) = 0$. 

The proof for $N^*$ is analogous. 
\end{proof}

Define
\begin{align*}
\mathfrak{TD} & 
\coloneqq (\{0,1,2\};C_3,C_2,T) \\
\text{ and } \quad \mathfrak{D} & \coloneqq (\{0,1,2\};C_3,C_2). 
\end{align*}
The following can be shown similarly as Proposition~\ref{prop:tl-l}. 

\begin{proposition}\label{prop:td-d}
The structures $\mathfrak{TD}$ and $\mathfrak{D}$ pp-construct each other. 
\end{proposition}
\begin{proof}
Since $\mathfrak{TD}$ is an expansion of $\mathfrak{D}$, it suffices to prove that $\mathfrak{D}$ pp-constructs $\mathfrak{TD}$. 
Let $\bA$ be the pp-power of $\mathfrak{D}$ with domain $\{0,1,2\}^2$ and the same signature as $\mathfrak{TD}$; the relations $C_3^{\mathfrak A}$ and $T^{\mathfrak A}$ are defined as in the proof of Proposition~\ref{prop:tl-l},  and 
\begin{align*}
C_2^{\mathfrak A}  \coloneqq \; & \{(x,y) \in A^2 \mid C_2(x_1,x_2) \wedge C_2(x_1,y_1) \wedge C_2(x_2,y_2) \} \\
 = \; & \{((0,1),(1,0)),((1,0),(0,1))\}.
\end{align*}
Let $g \colon \{0,1,2\} \to \{0,1,2\}^2$ and 
$h \colon \{0,1,2\}^2 \to  \{0,1,2\}$ be as in the proof of Proposition~\ref{prop:tl-l}. To verify that $h$ is a homomorphism from $\bA$ to 
$\mathfrak{TD}$, it suffices
to prove that $h$ maps tuples in $C_2^{\mathfrak A}$ to tuples in $C_2$, which is straightforward since $(0,1,1,0) \in C_2^{\bA}$ is mapped to 
$(0,1) \in C_2$ and 
$(1,0,0,1) \in C_2^{\bA}$ is mapped
to $(1,0) \in C_2$. Conversely, $g$ is a homomorphism from $\mathfrak{TD}$ to 
$\bA$. We have already shown that $g$ preserves $C_3$ and $T$ (see the proof of Proposition~\ref{prop:tl-l}). Finally, $(g(0),g(1)) = ((0,1),(1,0)) \in C_2^{\bA}$ and $(g(1),g(0)) = ((1,0),(0,1)) \in C_2^{\bA}$. 
\end{proof}

We conclude this section with a description of the clones of self-dual operations that collapse with the clone $\mathbf{C}_3$ of all self-dual operations. Define 
\begin{align*}
\bC_3 \coloneqq \; & (\{0,1,2\};C_3),  \\
\bC_3^{0} \coloneqq \; & (\{0,1,2\};C_3,\{0\}), 
\\\bC_3^{0,1} \coloneqq \; & (\{0,1,2\};C_3,\{0,1\}),\text{ and } \\
\bT \coloneqq \; & (\{0,1,2\};C_3,T) .
\end{align*}

\begin{proposition}
The structures $\bC_3^{0}$, $\bC_3$, and $\bT$ pp-construct each other. 
\end{proposition}
\begin{proof}
It follows from Corollary~\ref{cor:idempotentReduct} and Theorem~\ref{thm:wonderland} that $\bC_3^{0}$ and $\bC_3$ pp-construct each other. 
Since every relation of $\bC_3^{0}$ has a pp-definition in $\bT$,
it suffices to show that $\bC_3^{0}$ pp-constructs $\bT$.
This can be shown as in Proposition~\ref{prop:tl-l}.
\end{proof}

The next proposition requires a little more work.

\begin{proposition}\label{prop:conservativeCollapse}
The structure $\struct{C}_3$ pp-constructs $\struct{C}_3^{0,1}$.
\end{proposition}
\begin{proof}
We can use the constants $0$, $1$, and $2$ in our pp-formulas, see Section~\ref{sect:constants}. Analogously to the proof of Theorem~\ref{thm:collapse1}, we define a fourth pp-power $\struct{A}$ of $\struct{C}_3$: we consider $\struct{A}\coloneqq (\{0,1,2\}^4;C_3^{\struct{A}},R_{\{0,1\}}^{\struct{A}})$ where $C_3^{\struct{A}}$ is defined by the same primitive positive formula that defines $C_3^{\struct{P}'}$ in Theorem~\ref{thm:collapse1} and $R_{\{0,1\}}^{\struct{A}}$ is defined as follows:
\begin{align*}
    R_{\{0,1\}}^{\struct{A}} \coloneqq \big \{(x_0,x_1,x_2,x_3) \in \{0,1,2\}^4 \mid (x_0 = x_1) \wedge (x_3 = 1) \big \}.
\end{align*}
We define $g\colon \struct{C}_3^{0,1}\to \struct{A}$ as follows:
\begin{align*}
    g(0)\coloneqq (0,0,1,1); &&
    g(1)\coloneqq (1,1,0,1); &&
    g(2)\coloneqq (2,1,1,0).
\end{align*}
It is immediate to check that $g$ is a homomorphism.

Let
\begin{align*}
A_1 & \coloneqq\{(1,1,2),(1,2,1),(2,1,1)\} \\
\text{ and } A_2 & \coloneqq\{(0,1,1),(1,0,1),(1,1,0),(1,1,1)\}
\end{align*}
and define 
$h\colon \{0,1,2\}^4\to \{0,1,2\}$ as follows:
\begin{align*}
h(x_0,x_1,x_2,x_3) \coloneqq \begin{cases}
x_0 - 1 & \text{if }  (x_1,x_2,x_3) \in A_1 \\
x_0 & \text{if }  (x_1,x_2,x_3) \in A_2 \\
x_0 +1 & \text{if }  (x_1,x_2,x_3) \in A_3\coloneqq\{0,1,2\}^3\setminus(A_1\cup A_2).
\end{cases}
\end{align*}
\medskip  \noindent
\textbf{Claim.} $h$ is a homomorphism from $\struct{A}$ to $\struct{C}_3^{0,1}$.
\begin{itemize}
\item Let $(a,b)\in C_3^{\struct{A}}$. From the definition of $h$ it follows that if $(a_1,a_2,a_3)\in A_i$, for some $i\in\{1,2,3\}$, then $(b_1,b_2,b_3)\in A_i$. Therefore, 
    \begin{align*}
        \big(h(a_0,a_1,a_2,a_3),h(b_0,b_1,b_2,b_3)\big)\in \{(a_0-1,b_0-1),(a_0,b_0),(a_0+1,b_0+1)\}.
    \end{align*}
Since $(a_0,b_0)\in C_3$, it follows that $\big(h(a_0,a_1,a_2,a_3),h(b_0,b_1,b_2,b_3)\big)\in C_3$.
\item Let $(a_0,a_1,a_2,a_3)\in R_{\{0,1\}}$. Then $a_0=a_1$ and $a_3 = 1$. 
We distinguish three cases. If $a_0=0$, we have:
\begin{align*}
h(0,0,a_2,1) \coloneqq \begin{cases}
1 & \text{if } a_2 = 0,\\
0 & \text{if } a_2 = 1,\\
1 & \text{if } a_2 = 2.
\end{cases}
\end{align*}
If $a_0=1$:
\begin{align*}
h(1,1,a_2,1) \coloneqq \begin{cases}
1 & \text{if } a_2 = 0,\\
1 & \text{if } a_2 = 1,\\
0 & \text{if } a_2 = 2.
\end{cases}
\end{align*}
If $a_0=2$:
\begin{align*}
h(2,2,a_2,1) \coloneqq \begin{cases}
0 & \text{if } a_2 = 0,\\
1 & \text{if } a_2 = 1,\\
0 & \text{if } a_2 = 2.
\end{cases}
\end{align*}
Therefore, we can conclude that $h(a_0,a_1,a_2,a_3)\in \{0,1\}$.\qedhere
\end{itemize}
\end{proof}

Following our usual conventions, we define $\mathbf{C}^{0,1}_3\coloneqq \Pol(\struct{C}_3^{0,1})$ and $\mathbf{T}\coloneqq \Pol(\bT)$.

\begin{corollary}
The clones $\mathbf{C}_3$, $\mathbf{C}^0_3$, $\mathbf{C}^{0,1}_3$, and $\mathbf{T}$ are minor equivalent.
\end{corollary}

\section{Separations}
\label{sect:separation}
Here we prove that any two clones of self-dual operations whose minor-equivalence was not proved in the previous section are in fact not minor equivalent; 
for every pair $(\mathbf C,{\mathbf D})$ such that ${\mathbf C} \not \leq_{\operatorname{minor}} {\mathbf D}$   
we present a minor condition that holds in ${\mathbf C}$ but not in ${\mathbf D}$. 
Similarly as in Section~\ref{sect:collapse}, we proceed bottom-up in the lattice of clones  of self-dual operations (Figure~\ref{fig:Dima}). We first recall important minor conditions that we then use in later sections. 

\subsection{Some Important Minor Conditions}
\label{sect:mnor}
Unlike the guarded 3-cyclic condition from~(\ref{eq:guarded3cyclic}), the following minor conditions are well-known and omnipresent in universal algebra.  
The minor condition 
\begin{equation*}
    f(x, y)\approx f(y,x)
\end{equation*}
is called \emph{$2$-cyclic condition} and denoted by $\Sigma_2$. The minor condition  
\begin{equation*}
f(x,y,y) \approx f(y,y,x) \approx f(x,x,x)
\end{equation*}
is called \emph{quasi Mal'cev condition} and the minor condition 
\begin{equation*}
f(x,y,y) \approx f(y,x,y) \approx f(y,y,x) \approx f(x,x,x)
\end{equation*}
is called \emph{quasi minority condition}.
 Let $f$ be a $k$-ary function symbol. The minor condition 
\begin{equation*}
f(x,\dots,x,y) \approx f(x,\dots,y,x) \approx \dots \approx f(y,x,\dots,x) 
\end{equation*}
is called \emph{weak near-unanimity condition} $(\operatorname{WNU}(k))$. 
If we add to  $(\operatorname{WNU}(k))$ the minor identity 
$f(x,\dots,x,y) \approx f(x,\dots,x)$ then
the resulting condition is called \emph{quasi near-unanimity condition} $(\operatorname{QNU}(k))$. The $\operatorname{QNU}(3)$ condition is also called \emph{quasi majority condition}.
A $k$-ary operation $f$ is a \emph{quasi near-unanimity operation} if it satisfies the quasi near-unanimity condition $\operatorname{QNU}(k)$; we adopt an analogous convention for all other conditions introduced so far. A \emph{Mal'cev operation} is an idempotent quasi Mal'cev operation; we adopt a similar convention in defining a \emph{minority operation} and a \emph{majority operation}.

Now we define minor conditions with more than one operation symbol. 

\begin{definition}
Let $n \geq 2$. 
The minor condition 
\begin{align*}
t_0(x,y,z)& \approx t_0(x,x,x) \\ 
t_n(x,y,z)& \approx t_n(z,z,z) \\
t_i(x,y,x)& \approx t_i(x,x,x)& \mbox{ for } i \in \{0,\dots,n\} \\ 
t_i(x,x,z)& \approx t_{i+1}(x,x,z)& \mbox{for even } i \in \{0,\dots,n\} \\ 
t_i(x,z,z)& \approx t_{i+1}(x,z,z)& 
\mbox{for odd } i \in \{0,\dots,n\} 
\end{align*}
is called \emph{quasi J\'onsson condition of length $n$}, $\operatorname{QJ}(n)$.
\end{definition}

\begin{definition}
Let $n \geq 2$. 
The minor condition 
\begin{align*}
	p_0(x,y,z)& \approx p_0(x,x,x)
	\\ p_n(x,y,z)& \approx p_n(z,z,z), \text{ and }
	\\ p_i(x,x,y)& \approx p_{i+1}(x,y,y) \mbox{\ \ for every\ \ } i \in \{0,\dots, n-1\}
\end{align*}
is called \emph{quasi Hagemann-Mitschke condition of length $n$}, $\QHM(n)$.
\end{definition}

\subsection{The Atoms} 
\label{sect:atoms}
We show that there are four smallest clones of self-dual operations that, again with respect to $\leq_{\operatorname{minor}}$, are larger than $\Pol(\mathfrak{K}_3)$ (see Figure~\ref{fig:pict}). 

Let us define the following clones:
\begin{align*}
    \mathbf{L}_3 & \coloneqq \Pol(\{0,1,2\}; C_3,L_3)
    && \text{(see (\ref{def:RelationL3}))}\\ 
    \mathbf{TL}_2 & \coloneqq \Pol(\{0,1,2\}; C_3,T,L_2) 
    && \text{(see (\ref{def:T}) and (\ref{def:L}))}
    \\\mathbf{TN} & \coloneqq \Pol(\{0,1,2\};C_3,T,N,N^*) && \text{(see \eqref{eq:N} and \eqref{eq:dual})}
    \\\mathbf{W} & \coloneqq \Pol(\{0,1,2\}; C_3,R^=_3)&& \text{(see (\ref{def:relatioR3}))}. 
\end{align*}
The minority operation that returns $x$ whenever $|\{x,y,z\}|=3$ is denoted by 
`$\plus$'. 
The binary operation $\oplus$ 
is defined to be $(x,y) \mapsto 2(x+y) \mod 3$. 
It is known that $\mathbf{TL}_2 =  [\plus]$
and that $\mathbf{L}_3 =  [\oplus]$ (see~\cite{Zhuk15}, Theorem~28). 

\begin{figure}
\centering
\begin{tabular}{r|llll}
& ${\mathbf L}_3 \not \models$ & $\mathbf{TL}_2 \not \models$ & $\mathbf{TN}  \not \models$ & ${\mathbf W}  \not \models$ \\
\hline
${\mathbf L}_3 \models $ & & $\Sigma_2$ & $\Sigma_2$ & \text{Mal'cev} \\
$\mathbf{TL}_2 \models $ & \text{minority} & &  \text{minority} & \text{minority} \\
$\mathbf{TN} \models$ & $\text{majority}$ & $\text{majority}$ & & $\text{majority}$ \\
${\mathbf W} \models$ & $\WNU(3)$ & $\Sigma_2$ & $\Sigma_2$ & 
\end{tabular}
\caption{The minor conditions that show that the clones $\mathbf{L}_3$, $\mathbf{TL}_3$, $\mathbf{TN}$, and $\mathbf{W}$ are pairwise incomparable.}
\label{fig:sep}
\end{figure}

\begin{proposition}\label{Prop:Atoms}
The clones $\mathbf{L}_3$, $\mathbf{TL}_2$, $\mathbf{TN}$, and $\mathbf{W}$ are pairwise incomparable with respect to ${\leq}_{\operatorname{minor}}$. 
\end{proposition}
\begin{proof}
We use the minor conditions as specified in Figure~\ref{fig:sep}. 
We claim that $\Sigma_2$ holds in $\mathbf{L}_3$ and in ${\mathbf W}$ but not in $\mathbf{TL}_2$ and $\mathbf{TN}$.
Clearly, the operation $\oplus$ in ${\mathbf L}_3$  and the operation $\vee_3$ in ${\mathbf W}$ are 2-cyclic. 
Every 2-cyclic operation that preserves $\{0,1\}$ does not preserve the relation $T$, because $f(0,1)=f(1,0)=a\in\{0,1\}$ and $(a,a)\notin T$. Note that all operations in $\mathbf{TL}_2$ or in $\mathbf{TN}$ preserve $\{0,1\}$.

It is easy to check that the operation $(x,y,z) \mapsto x \oplus (y \oplus z) \in\mathbf{L}_3$ 
is a quasi Mal'cev operation. But any Mal'cev operation $f$
does not preserve the relation $R^=_3$
because 
\begin{align*}
\begin{pmatrix} 
f(0,1,1)\\
f(1,1,1)\\
f(1,1,0)
\end{pmatrix} = 
\begin{pmatrix} 
0\\
1\\
0
\end{pmatrix}\notin R^=_3.
\end{align*} 
which shows that ${\mathbf W}$ does not satisfy the quasi Mal'cev condition.

Note that $\mathbf{TL}_2=[\operatorname{plus}]$ satisfies the quasi minority condition but $\mathbf{L}_3$, $\mathbf{TN}$, and ${\mathbf W}$ do not. To see this, let $f$ be a quasi minority operation $f$. Then $f$ does not preserve $L_3$ because  
\begin{align*}
\begin{pmatrix} 
f(0,0,1)\\
f(0,1,0)\\
f(0,2,2)
\end{pmatrix} = 
\begin{pmatrix} 
1\\
1\\
0
\end{pmatrix}\notin L_3. 
\end{align*} 
Moreover, $f$ does not preserve $N$ because 
\[\begin{pmatrix} 
f(0,1,1)\\
f(0,0,2)
\end{pmatrix} = 
\begin{pmatrix} 
0\\
2
\end{pmatrix}\notin N.\]
Finally, we have already seen that ${\mathbf W}$ does not have a quasi minority operation because every quasi minority operation is a particular quasi Mal'cev operation. 

The clone $\mathbf{TN} = [m]$ (see (\ref{def:m}))  satisfies $\operatorname{QNU}(3)$ since $m$ is a quasi majority operation. However, any quasi majority operation $f$ does not preserve $L_3$ as
\begin{align*}
\begin{pmatrix} 
f(0,0,1)\\
f(0,1,0)\\
f(0,2,2)
\end{pmatrix} = 
\begin{pmatrix} 
0\\
0\\
2
\end{pmatrix}\notin L_3,
\end{align*}
does not preserve $L_2$ as
\[\begin{pmatrix}
f(0,0,1)
\\f(0,1,1)
\\f(0,1,0)
\end{pmatrix}=
\begin{pmatrix}
0
\\1
\\0
\end{pmatrix}\notin L_2,\]
and does not preserve $R^=_3$ as
\begin{align*}
\begin{pmatrix}
f(0,0,1)
\\f(0,1,1)
\\f(0,1,0)
\end{pmatrix} & =\begin{pmatrix}
0
\\1
\\0
\end{pmatrix}\notin R^=_3,
\end{align*}
which shows that ${\mathbf L}_3$, $\mathbf{TL}_2$, and $\mathbf W$ do not satisfy quasi majority. 
Note that $\operatorname{WNU}(3)$ holds in $\mathbf W$ since $\mathbf W$ contains $(x,y,z) \mapsto x \vee_3 (y \vee_3 z)$.
Suppose that there exists a ternary weak near unanimity operation $w \in \mathbf{L}_3$. Then 
\begin{align*}
\begin{pmatrix} 
w(0,1,1)\\
w(1,0,1)\\
w(2,2,1)
\end{pmatrix} = 
\begin{pmatrix} 
a \\
a \\
b 
\end{pmatrix}.
\end{align*} 
Note that $(a,a,b) \in L_3$ implies that $a=b$. 
It follows that 
\begin{align*}
\begin{pmatrix} 
w(0,1,1)\\
w(1,2,2)
\end{pmatrix} = 
\begin{pmatrix} 
a \\
a 
\end{pmatrix} \notin C_3
\end{align*} 
which is a contradiction.
\end{proof}

\subsection{Separations in the Wings} 
In the lattice of clones of self-dual operations the clone
$\mathbf{W}$ is the smallest clone in the left wing. 
The clone ${\mathbf Q} \coloneqq \Pol(\mathfrak{Q})$, where ${\mathfrak Q} = (\{0,1,2\};C_3,R^=_2)$ is the structure introduced in Section~\ref{sect:pq-collapse}, 
is the unique smallest clone that properly contains
$\mathbf{W}$ and also lies in the left wing (see Figure~\ref{fig:Dima}). 

\begin{theorem}\label{thm:gsigma3}
There is no minor-preserving map from ${\mathbf Q}$ to
${\mathbf W}$. 
\end{theorem}
\begin{proof}
It is known (\cite{Zhuk15}, Theorem~29) that the clone ${\mathbf Q}$ contains the operation $r_4$ defined as follows:
\begin{align*}
r_4(x,y,z,t) \coloneqq \begin{cases}
x \vee_3 y \vee_3 z & \text{ if } |\{x,y,z\}| \leq 2 \\
t & \text{ otherwise.}
\end{cases}
\end{align*}
Note that $r_4$ satisfies the minor condition, which we call \emph{guarded 3-cyclic ($g\Sigma_3$)}
\begin{align}
f(x_1,x_2,x_3,y) & \approx f(x_2,x_3,x_1,y) \label{eq:cycl} \\
\text{ and } f(x,x,x,y) & \approx f(x,x,x,x). \label{eq:guard}
\end{align}
Suppose for contradiction that these identities can be satisfied by an operation $f \in \Pol({\mathfrak W})$. Let $a \in \{0,1,2\}$ be such that $f(0,1,2,0)=a$. Since $f$ preserves $C_3$ we have $f(1,2,0,1) = a+1$. By~(\ref{eq:cycl}) we have
$f(0,1,2,1) = a+1$. But then $f$ does not preserve
$R^=_3$, because 
\begin{align*}
f \begin{pmatrix}
0 & 0 & 0 & 1 \\
0 & 1 & 2 & 0 \\
0 & 1 & 2 & 1 
\end{pmatrix}  = 
\begin{pmatrix} 
0 \\ 
a \\ 
a+1 
\end{pmatrix} 
\notin R^=_3.&\qedhere
\end{align*} 
\end{proof}

\begin{proposition}\label{prop:qj}
The minor condition $\QJ(4)$ holds in ${\mathbf M}_{\infty}$ but not in ${\mathbf Q}$. 
\end{proposition}
\begin{proof}
Zhuk (\cite{Zhuk15}, Theorem~30) proved  that ${\mathbf M}_{\infty}$ contains the operation $f^{\infty}_\pi$ defined by 
\[f^{\infty}_\pi (x,y,z) \coloneqq  \begin{cases}
x \vee_3 (y \wedge_3 z) & \text{ if } |\{x,y,z\}| \leq 2 \\
x & \text{ otherwise.}
\end{cases}
\]
Note that the operations $t_0,t_1,t_2,t_3,t_4$ given by 
\begin{align*}
t_0(x,y,z) & \coloneqq f^{\infty}_\pi(x,x,x) \\
t_1(x,y,z) & \coloneqq f^{\infty}_\pi(x,y,z) \\
t_2(x,y,z) & \coloneqq f^{\infty}_\pi(x,z,z) \\
t_3(x,y,z) & \coloneqq f^{\infty}_\pi(z,x,y) \\
t_4(x,y,z) & \coloneqq f^{\infty}_\pi(z,z,z) 
\end{align*}
witness that ${\mathbf M}_{\infty}$ satisfies $\QJ(4)$; in particular, we have 
\begin{align*}
t_1(x,y,x) & = f^{\infty}_\pi(x,y,x) = x \vee_3 (y \wedge_3 x) = x \\
t_2(x,y,x) & = f^{\infty}_\pi(x,x,x) = x \\
t_3(x,y,x) & = f^{\infty}_\pi(x,x,y) = x \vee_3 (x \wedge_3 y) = x  \\
t_1(x,z,z) & = f^{\infty}_\pi(x,z,z) = f^{\infty}_{\pi}(x,z,z) = t_2(x,z,z) \\
t_2(x,x,z) & = f^{\infty}_\pi(x,z,z) = f^{\infty}_{\pi}(z,x,x) = t_3(x,x,z) . 
\end{align*}
We claim that $\mathbf Q$ does not satisfy
$\QJ(4)$. Since $\bf Q$ has a minor-preserving map to 
$\bf P$, it suffices to prove that the clone $\bf P$ does not satisfy $\QJ(4)$. 
Barto~\cite{barto-cd} proved that every finite structure
with a finite relational signature whose polymorphism clone satisfies $\QJ(n)$ for some $n$ also satisfies $\QNU(n)$ for some $n \geq 3$. 
Suppose for contradiction that
$\bf P$ contains a quasi near unanimity operation $t$ of arity $n \geq 3$. 
If $n = 3$ we have 
\begin{align*}
    \begin{pmatrix}
    t(0,1,0)
    \\t(1,0,0)
    \\t(1,1,0)
\end{pmatrix}=
 \begin{pmatrix}
    0
    \\0
    \\1
\end{pmatrix}\notin R^\Rightarrow_2.
\end{align*}
For $k > 3$, note that 
\begin{align*}
    \begin{pmatrix}
    t(0,\dots,0,1,0)
    \\t(1,\dots,1,0,0)
    \\t(1,\dots,1,1,0)
\end{pmatrix}=
 \begin{pmatrix}
    0
    \\a_1
    \\1
\end{pmatrix}.
\end{align*}
Since we assumed that $t\in\Pol(\mathfrak{P)}$, we obtain that $(0,a_1,1)\in R^\Rightarrow_2$ and so $a_1=1$. 
We repeat the same reasoning $k-4$ times: each time we consider the matrix $M_i$, for $2\leq i\leq k-3$, where the first $k-i-1$ columns are equal to $(0,1,1)$, the ($k-i$)-th column is equal to $(1,0,1)$, and the last $i$ columns are equal to $(0,0,0)$. Note that every column of $M_i$ is an element of $R^\Rightarrow_2$. By letting $t$ act  row-wise in $M_i$, we get the chain of equalities $a_2=\dots = a_{k-3}=1$. Finally we get
\begin{align*}
    \begin{pmatrix}
    t(0,1,0,\dots,0)
    \\t(1,0,0,\dots,0)
    \\t(1,1,0,\dots,0)
\end{pmatrix}=
 \begin{pmatrix}
    0
    \\0
    \\a_{k-3}
\end{pmatrix}\notin R^\Rightarrow_2,
\end{align*}
a contradiction.
\end{proof}

Recall that in Section~\ref{sect:pq-collapse} we proved
that ${\mathbf P}$ and ${\mathbf Q}$ are minor equivalent and, in Section~\ref{sect:maincollapse}, that ${\mathbf B}_{\infty}\pi_\infty$ collapses with $\mathbf M_{\infty}$. It follows that $\mathbf M_{\infty}$ is the unique smallest element that properly contains ${\mathbf Q}$ in Figure~\ref{fig:pict}.

We will now show all the clones 
${\mathbf B}_k$ and ${\mathbf M}_k$, for $k \geq 2$, are pairwise distinct. 
If $t \in A^n$, we denote by
$\Two(t)$ the set of all entries of $t$ that appear at least twice. 
For $n \geq 3$, let $f_\pi^n \colon \{0,1,2\}^{n+1} \to \{0,1,2\}$ be the operation defined as follows
\begin{align*}
    f_{\pi}^{n}(x_1,\dots,x_{n+1})& \coloneqq \begin{cases}
    x_1 & \text{ if } \Two(x_1,\dots,x_{n+1})=\{0,1,2\},
    \\a\vee_3 b & \text{ if } \Two(x_1,\dots,x_{n+1})=\{a,b\},
    \\a & \text{ if } \Two(x_1,\dots,x_{n+1})=\{a\}. 
    \end{cases}
    \end{align*}
It is known that $f_\pi^{n} \in {\mathbf B}_n$ and $f_\pi^n \in {\mathbf M}_n$ 
(\cite{Zhuk15}, Theorem~29). 

\begin{proposition}\label{prop:qnu-sep}
Let $n \geq 2$. The condition 
$\QNU(n+1)$
holds in ${\mathbf B}_n$ and ${\mathbf M}_n$, 
but not in ${\mathbf B}_{n+1}$ and ${\mathbf M}_{n+1}$. 
\end{proposition}
\begin{proof}
If $n \geq 3$, then ${\mathbf B}_n$ and ${\mathbf M}_n$ contain the quasi near unanimity operation $f_\pi^{n}$, which is a witness for ${\mathbf B}_n\models\QNU(n+1)$ and ${\mathbf M}_n\models\QNU(n+1)$. 
If $n=2$, then $m \in {\mathbf B}_n$
and $m \in {\mathbf M}_n$ (where $m$ is defined in~(\ref{def:m})) and that $m$ is a majority operation. 
It is easy to see that every quasi near unanimity operation of arity $n \geq 3$ does not preserve $B_n$, so
$\QNU(n)$ does not hold in ${\mathbf B}_{n}$ and in ${\mathbf M}_{n}$. 
\end{proof} 

Let $f_0^{\infty} \colon \{0,1,2\}^3 \to \{0,1,2\}$ be defined by $f_0^{\infty}(x,y,x) = x \vee_3 y$, and by $f_0^{\infty}(x,y,z) = x$ otherwise. 
It is known (\cite{Zhuk15}, Theorem~29) that $f_0^{\infty} \in {\mathbf B}_n$, for every $n \geq 2$. 

\begin{proposition}\label{prop:qhm-sep}
$\QHM(3)$ holds in ${\mathbf B}_\infty$, but not in 
${\mathbf M}_2$. 
\end{proposition}
\begin{proof}
For $x,y,z \in \{0,1,2\}$, define
\begin{align*}
p_0(x,y,z) & \coloneqq f_0^{\infty}(x,x,x) &
p_1(x,y,z) & \coloneqq f_0^{\infty}(x,z,y) \\
p_2(x,y,z) & \coloneqq f_0^{\infty}(z,x,y) &
p_3(x,y,z) & \coloneqq f_0^{\infty}(z,z,z).
\end{align*}
Then $p_0,p_1,p_2,p_3$ witness that ${\mathbf B}_{\infty}$ satisfies $\QHM(3)$: in particular, we have 
\begin{align*}
p_1(x,x,y) = f_0^{\infty}(x,y,x) = x\vee_3 y  = f_0^{\infty}(y,x,y) = p_2(x,y,y). 
\end{align*}
Suppose for contradiction that ${\mathbf M}_2$ has
operations $p_0,p_1,p_2,p_3$ that witness $\QHM(3)$. Then 
\begin{align*} 1 = p_0(1,1,0) & = p_1(1,0,0) \\
& \leq_2 p_1(1,1,0) = p_2(1,0,0) \\
& \quad \quad \quad \quad \quad \! \quad \leq_2 p_2(1,1,0) = p_3(1,0,0) = 0
\end{align*}
which is a contradiction. 
\end{proof}

This implies that for all $k,l \geq 2$,
there is no minor-preserving map from
${\mathbf B}_k$ to ${\mathbf M}_l$, because  ${\mathbf B}_\infty \subseteq {\mathbf B}_k$
and ${\mathbf M}_l \subseteq {\mathbf M}_2$.

\subsection{The Final Picture}
In this section we complete the proof of the following theorem. 

\begin{figure}
\begin{footnotesize}
\centering
\begin{tabular}{r|lllllllll}
& ${\mathbf L}_3 \not \models$ & $\mathbf{TL}_2 \not \models$ & $\mathbf{TN}  \not \models$ & ${\mathbf W}  \not \models$ & ${\mathbf D} \not \models$ & ${\mathbf Q} \not \models$ & ${\mathbf M}_n \not \models$ & ${\mathbf B}_n \not \models$ \\
\hline
${\mathbf L}_3 \models $ & & $\Sigma_2$ & $\Sigma_2$ & \text{Mal'cev} & $\Sigma_2$ & \text{Mal'cev} & \text{Mal'cev} & \text{Mal'cev} \\
$\mathbf{TL}_2 \models $ & \text{minority} & &  \text{minority} & \text{minority} & & \text{minority} & \text{minority} & \text{minority} \\
$\mathbf{TN} \models$ & $\text{majority}$ & $\text{majority}$ & & $\text{majority}$ & & $\text{majority}$ & $\text{majority}$ & $\text{majority}$ \\
${\mathbf W} \models$ & $\WNU(3)$ & $\Sigma_2$ & $\Sigma_2$ & & $\Sigma_2$ & \\
${\mathbf D} \models$ & $\text{majority}$ & $\text{majority}$ & $\text{minority}$ & $\text{minority}$ & & $\text{minority}$  & $\text{minority}$  & $\text{minority}$  \\
${\mathbf Q} \models$ & $\WNU(3)$ & $\Sigma_2$ & $\Sigma_2$ & $g\Sigma_3$ & $\Sigma_2$ \\
${\mathbf M}_n \models$ & $\WNU(3)$ & $\Sigma_2$ & $\Sigma_2$ & $g\Sigma_3$ & $\Sigma_2$ & $\QJ(4)$ \\
${\mathbf B}_n \models$  & $\WNU(3)$ & $\Sigma_2$ & $\Sigma_2$ & $g\Sigma_3$ & $\Sigma_2$ & $\QJ(4)$ & $\QHM(3)$ \\
${\mathbf C}_3 \models$ & $\text{minority}$ & $\Sigma_2$ & $ \text{minority}$ & \text{minority} & $\Sigma_2$ & \text{minority} & \text{minority} & \text{minority} 
\end{tabular}
\end{footnotesize}
\caption{The minor conditions that justify that the existence of a minor-preserving map orders the clones of self-dual operations as depicted in Figure~\ref{fig:pict}.}
\label{fig:sep2}
\end{figure}

\begin{theorem}\label{thm:main}
The lattice of clones  of self-dual operations  factored by minor equivalence and ordered by the existence of minor-preserving maps, is a countably infinite lattice, 
and is exactly of the form as described in Figure~\ref{fig:pict}. 
\end{theorem}
\begin{proof}
We use the minor conditions as indicated in 
the table of Figure~\ref{fig:sep2}. 
In Proposition~\ref{prop:qnu-sep}
and Proposition~\ref{prop:qhm-sep} we proved that the clones
${\mathbf M}_n$ and ${\mathbf B}_n$, for $n \in \{2,3,\dots,\infty\}$, form two descending chains as displayed in Figure~\ref{fig:sep2}. 
The restriction of this table to the clones
${\mathbf L}_3$, $\mathbf{TL}_2$, ${\mathbf T}$, and ${\mathbf T}$ has already been described in Figure~\ref{fig:sep}. We now describe how to extend this table to the remaining clones
${\mathbf W}$, ${\mathbf D} \coloneqq \Pol(\bD)$,
${\mathbf Q}$, ${\mathbf M}_n$, ${\mathbf B}_n$, and ${\mathbf C}_3$. 

The clone ${\mathbf D} = \Pol(\{0,1,2\};C_3,C_2)$ contains $\mathbf{TN}$ and $\mathbf{TL}_2$, 
and therefore contains the minority operation $\plus$ and the majority operation $m$. 
It does not 
satisfy $\Sigma_2$ because operations on $\{0,1,2\}$ that satisfy $\Sigma_2$ cannot preserve $C_2$. 

The clone ${\mathbf Q}$ contains ${\mathbf W}$ and therefore contains a binary symmetric operation and a ternary weak near unanimity operation. Moreover, the proof of Theorem~\ref{thm:gsigma3} shows that 
${\mathbf Q}$ satisfies $g\Sigma_3$, which is not satisfied by ${\mathbf W}$. 
Moreover, ${\mathbf Q}$ does not satisfy $\QJ(4)$, which is satisfied by ${\mathbf M}_n$ and ${\mathbf B}_n$ for all $n \in \{2,3,\dots,\infty\}$ (Proposition~\ref{prop:qj}). 

The clone ${\mathbf M}_n$, for each $n \in \{2,3,\dots,\infty\}$, contains ${\mathbf Q}$ and therefore satisfies $\QNU(3)$, $\Sigma_2$, and $g\Sigma_3$. It does not satisfy $\QHM(3)$ as we have seen in Proposition~\ref{prop:qhm-sep}. 
For each $n \in \{2,3,\dots,\infty\}$, 
the clone ${\mathbf B}_n$ contains 
${\mathbf M}_n$, but satisfies the additional minor condition $\QHM(3)$. It is straightforward to verify that any minority operation on $\{0,1,2\}$ does not preserve the relation $B_2$, so 
${\mathbf B}_n$ does not satisfy the minority condition. The clone ${\mathbf C}_3$ contains all the clones discussed so far; since each of these clones does not satisfy some minor condition discussed so far, it follows that 
${\mathbf C}_3$ does not have a minor-preserving map to any of these clones. 
\end{proof}

\section{Concluding Remarks}
There are $2^\omega$ many clones of self-dual operations;
we showed that there is an equally large set of clones  of self-dual operations that are pairwise homomorphically incomparable. However,
when considered up to minor equivalence the self-dual clones fall into countably many classes, and the resulting lattice can be described completely. 
Our work illustrates the potential of minor equivalence for a systematic study of clones on finite domains. We make the following provocative conjecture. 

\begin{conjecture}
The minor equivalence relation on clones on finite domains has only countably many classes. 
\end{conjecture}

An interesting first step towards proving this conjecture would be the verification that 
there are only countably many clones on a three-element set up to minor equivalence. 

All clones that contain an operation with an image of size at most two are minor equivalent to a clone on a two-element set  (these clones comprise the uncountable set of clones described by Yanov and Muchnik~\cite{YanovMuchnik}), and it is known that there are only countably many of those by Post's classification. 
By the discussion in Section~\ref{sect:constants}, 
it therefore suffices to classify \emph{idempotent} clones on three elements with respect to minor equivalence. 

\begin{conjecture}\label{conj:maximal}
Every clone on three elements that does not admit a minor-preserving map from 
the clone ${\mathbf I}$ of all idempotent operations on $\{0,1,2\}$ 
has a minor-preserving map to one of the following three clones on $\{0,1,2\}$: 
\begin{enumerate}
\item the clone ${\mathbf C}_2 \coloneqq \Pol(\{0,1\};\{(0,1),(1,0)\})$,
\item the clone ${\mathbf C}_3$, whose subclones were studied in this article, 
\item the clone ${\mathbf T}_3 \coloneqq \Pol(\{0,1,2\};<)$.
\end{enumerate}
\end{conjecture}
The three clones in Conjecture~\ref{conj:maximal} also appear as polymorphism clones of directed graphs that are maximal with respect to pp-constructability in the class of all finite directed graphs~\cite{maximal-digraphs}. 
If Conjecture~\ref{conj:maximal} is true, then the results of the present article show that in order to classify all clones over $\{0,1,2\}$ up to minor-equivalence it suffices to study the clones that have a minor-preserving map to  
${\mathbf C}_2$ or to 
${\mathbf T}_3$. 
If we are only interested in proving that there are countably many such clones, we can even focus on the clones with a minor-preserving map to ${\mathbf T}_3$, because all other clones have a Mal'cev operation~\cite{maximal-digraphs}
and because Bulatov~\cite{BulatovOnTheNumber} proved that there are only finitely many clones on $\{0,1,2\}$ containing a Mal'cev operation. Aichinger, Mayr, and McKenzie~\cite{AichiMayrMcKenzie} proved more recently that there are only countably many clones on sets of the form $\{0,\dots,n-1\}$ containing a Mal’cev operation.

\section*{Acknowledgements}
The authors thank the referees for the many very good comments that helped to improve the readability of the proofs.

\begin{funding}
The first two authors have received funding from the European Research Council (ERC Grant
Agreement no. 681988, CSP-Infinity). The third author has received funding 
from the European Research Council (ERC Grant
Agreement no. 771005, CoCoSym).
\end{funding}

\bibliographystyle{abbrv}
\def\cprime{$'$} \def\cprime{$'$} \def\cprime{$'$}

\end{document}